\documentclass[11pt,reqno]{amsart}
\usepackage{amsfonts,amsbsy,amsmath,amssymb, amsthm}
\usepackage{verbatim} 
\usepackage{eufrak}
\usepackage{upgreek}
\usepackage{graphics}
\usepackage{color}
\usepackage{bbm}
\usepackage[usenames,dvipsnames]{xcolor}
\usepackage{enumitem}
\usepackage{mathtools}
\usepackage{fullpage}
\usepackage[T1]{fontenc} 
\usepackage{lipsum}                     
\usepackage{xargs} 
\usepackage[normalem]{ulem}

\definecolor{myMaroon}{HTML}{720E0E}
\definecolor{myBlue}{HTML}{1A5276}
\definecolor{myDarkerBlue}{HTML}{154360}

\usepackage[colorlinks=true, citecolor=myBlue, urlcolor = purple, linkcolor=myMaroon, anchorcolor = myBlue]{hyperref}

\usepackage[colorinlistoftodos,prependcaption,textsize=footnotesize]{todonotes}

\newcommandx{\unsure}[2][1=]{\todo[linecolor=red,backgroundcolor=red!80,bordercolor=red,#1]{#2}}
\newcommandx{\change}[2][1=]{\todo[linecolor=blue,backgroundcolor=blue!25,bordercolor=blue,#1]{#2}}
\newcommandx{\info}[2][1=]{\todo[linecolor=OliveGreen,backgroundcolor=OliveGreen!25,bordercolor=OliveGreen,#1]{#2}}
\newcommandx{\improve}[2][1=]{\todo[linecolor=Plum,backgroundcolor=Plum!25,bordercolor=Plum,#1]{#2}}
\newcommandx{\TODO}[2][1=]{\todo[linecolor=myMaroon,backgroundcolor=myMaroon!25,bordercolor=myMaroon,#1]{#2}}
\newcommandx{\TODOIN}[2][1=]{\todo[inline,linecolor=myMaroon,backgroundcolor=myMaroon!25,bordercolor=myMaroon,#1]{#2}}
\newcommandx{\unsurein}[2][1=]{\todo[inline,linecolor=red,backgroundcolor=red!85,bordercolor=myMaroon,#1]{#2}}

\numberwithin{equation}{section}
\newcommand{\ZZ}{\mathbb{Z}}
\newcommand{\slanted}[1]{\slshape{#1}}

\newcommand{\scaps}[1]{{\scshape #1}}
\newcommand{\bscaps}[1]{\textsc{\textbf{#1}}}
\newcommand{\bfm}{\textbf}
\newcommand{\mcal}{\mathcal}
\newcommand{\mbb}{\mathbb}
\newcommand{\mbs}{\boldsymbol}
\newcommand{\mand}{\; \text{and}\;}

\newcommand{\PP}{\mbb{P}}
\newcommand{\dcup}{\dot\cup}
\newcommand{\tA}{\tilde{A}}
\newcommand{\Pro}{\mbb{P}}
\newcommand{\Ex}{\mbb{E}}
\newcommand{\Naturals}{\mbb{N}}
\newcommand{\NN}{\mbb{N}}
\newcommand{\Reals}{\mbb{R}}

\newcommand{\sm}{\setminus}
\newcommand{\discup}{\dot{\cup}}

\newcommand{\cHHI}{\cH_I}
\newcommand{\cHHW}{\cH_W}
\newcommand{\ovI}{\overline{I}}

\newcommand{\ind}{\mathbbm{1}}

\def\rk{\mathrm{\mbs{rk}}\,}

\def\bpi{\boldsymbol{\pi}}

\def\bi{\boldsymbol{i}}
\def\br{\boldsymbol{r}}

\def\bH{\boldsymbol{H}}

\def\bfH{\boldsymbol{\fH}}

\let\eps=\varepsilon
\let\theta=\vartheta
\let\rho=\varrho
\let\sigma=\varsigma
\let\phi=\varphi

\def\QED{$\blacksquare$}
\def\inQED{$\square$}

\renewenvironment{proof}
{\vspace{1ex}\noindent{\slanted Proof.}\hspace{0.5em}}{\hfill \QED \vspace{1ex}}

\newenvironment{theorem}
{
\refstepcounter{equation} 
\vspace{-1ex}
\ \\
\noindent
\begin{it}
\noindent
\bscaps{Theorem~\theequation.}\hspace{-0.5ex}
}
{\end{it} \vspace{1ex}}

\newenvironment{lemma}
{
\refstepcounter{equation} 
\vspace{-1ex}
\ \\
\noindent
\begin{it}
\noindent
\bscaps{Lemma~\theequation.}\hspace{-0.5ex}
}
{\end{it} \vspace{1ex}}

\newenvironment{observation}
{
\refstepcounter{equation} 
\vspace{-1ex}
\ \\
\noindent
\begin{it}
\noindent
\bscaps{Observation~\theequation.}\hspace{-0.5ex}
}
{\end{it} \vspace{1ex}}

\newenvironment{corollary}
{
\refstepcounter{equation} 
\vspace{-1ex}
\ \\
\noindent
\begin{it}
\noindent
\bscaps{Corollary~\theequation.}\hspace{-0.5ex}
}
{\end{it} \vspace{1ex}}

\newenvironment{definition}
{
\refstepcounter{equation} 
\vspace{-1ex}
\ \\
\noindent
\begin{it}
\noindent
\bscaps{Definition~\theequation.}\hspace{-0.5ex}
}
{\end{it} \vspace{1ex}}

\newenvironment{conjecture}
{
\refstepcounter{equation} 
\vspace{-1ex}
\ \\
\noindent
\begin{it}
\noindent
\bscaps{Conjecture~\theequation.}\hspace{-0.5ex}
}
{\end{it} \vspace{1ex}}

\newenvironment{remark}
{
\refstepcounter{equation} 
\vspace{-1ex}
\ \\
\noindent
\bscaps{Remark~\theequation.}\hspace{-0.5ex}
}
{\vspace{1ex}}

\newcommand{\A}{\mcal{A}}

\newcommand{\C}{\mcal{C}}
\newcommand{\cC}{\mcal{C}}

\newcommand{\F}{\mcal{F}}

\newcommand{\cH}{\mcal{H}}

\renewcommand{\P}{\mcal{P}}
\newcommand{\cP}{\mcal{P}}

\newcommand{\cT}{\mcal{T}}

\newcommand{\W}{\mcal{W}}

\newcommand{\fH}{\mathfrak{H}}

\newcommand{\fU}{\mathfrak{U}}

\makeatletter
\def\section{\@ifstar\unnumberedsection\numberedsection}
\def\numberedsection{\@ifnextchar[
  \numberedsectionwithtwoarguments\numberedsectionwithoneargument}
\def\unnumberedsection{\@ifnextchar[
  \unnumberedsectionwithtwoarguments\unnumberedsectionwithoneargument}
\def\numberedsectionwithoneargument#1{\numberedsectionwithtwoarguments[#1]{#1}}
\def\unnumberedsectionwithoneargument#1{\unnumberedsectionwithtwoarguments[#1]{#1}}
\def\numberedsectionwithtwoarguments[#1]#2{%
  \ifhmode\par\fi
  \removelastskip
  \vskip 1.7ex\goodbreak
  \refstepcounter{section}%
  \begingroup
  \noindent\leavevmode\Large\bfseries\scshape\normalsize
  \begin{center}\S \thesection.\ #2\end{center} 
  \endgroup
  \addcontentsline{toc}{section}{%
    \protect\numberline{\bfm{\thesection.}}%
    \hspace{2.5ex} #1}%
  }
\def\unnumberedsectionwithtwoarguments[#1]#2{%
  \ifhmode\par\fi
  \removelastskip
  \vskip 1.7ex\goodbreak
  \begingroup
  \noindent\leavevmode\Large\bfseries\scshape\centering 
  \begin{center} #2 \end{center} \par
  \endgroup
  \vskip 2ex\nobreak
  \addcontentsline{toc}{section}{%
    \hspace{1ex} #1}%
  }
\makeatother

\makeatletter
\def\subsection{\@ifstar\unnumberedsubsection\numberedsubsection}
\def\numberedsubsection{\@ifnextchar[
  \numberedsubsectionwithtwoarguments\numberedsubsectionwithoneargument}
\def\unnumberedsubsection{\@ifnextchar[
  \unnumberedsubsectionwithtwoarguments\unnumberedsubsectionwithoneargument}
\def\numberedsubsectionwithoneargument#1{\numberedsubsectionwithtwoarguments[#1]{#1}}
\def\unnumberedsubsectionwithoneargument#1{\unnumberedsubsectionwithtwoarguments[#1]{#1}}
\def\numberedsubsectionwithtwoarguments[#1]#2{%
  \ifhmode\par\fi
  \removelastskip
  \vskip 1.7ex\goodbreak
  \refstepcounter{subsection}%
  \noindent
  \leavevmode
  \begingroup
  \bfseries\normalsize
  \noindent\S \thesubsection\ \bscaps{#2.}\ 
  \endgroup
  \addcontentsline{toc}{subsection}{%
    \hspace{2ex}\protect\numberline{\bfm{\thesubsection.}}%
    \hspace{1ex} #1}%
  }
\def\unnumberedsubsectionwithtwoarguments[#1]#2{%
  \ifhmode\par\fi
  \removelastskip
  \vskip 3ex\goodbreak
  \noindent
  \leavevmode
  \begingroup
  \bfseries\normalsize
  \begin{center}\bscaps{#2.} \end{center}
  \endgroup
  \addcontentsline{toc}{subsection}{%
    \hspace{1ex} #1}%
  }
\makeatother

\makeatletter
\def\subsubsection{\@ifstar\unnumberedsubsubsection\numberedsubsubsection}
\def\numberedsubsubsection{\@ifnextchar[
  \numberedsubsubsectionwithtwoarguments\numberedsubsubsectionwithoneargument}
\def\unnumberedsubsubsection{\@ifnextchar[
  \unnumberedsubsubsectionwithtwoarguments\unnumberedsubsubsectionwithoneargument}
\def\numberedsubsubsectionwithoneargument#1{\numberedsubsubsectionwithtwoarguments[#1]{#1}}
\def\unnumberedsubsubsectionwithoneargument#1{\unnumberedsubsubsectionwithtwoarguments[#1]{#1}}
\def\numberedsubsubsectionwithtwoarguments[#1]#2{%
  \ifhmode\par\fi
  \removelastskip
  \vskip 3ex\goodbreak
  \refstepcounter{subsubsection}%
  \noindent
  \leavevmode
  \begingroup
  \bfseries
  \S \thesubsubsection\ \bscaps{#2.}\  
  \endgroup
  \addcontentsline{toc}{subsubsection}{%
    \hspace{2ex} \protect\numberline{\bfm{\thesubsubsection.}}%
    #1}%
  }
\def\unnumberedsubsubsectionwithtwoarguments[#1]#2{%
  \ifhmode\par\fi
  \removelastskip
  \vskip 3ex\goodbreak
  \noindent
  \leavevmode
  \begingroup
  \bfseries
  \bscaps{#2.}\  
  \endgroup
  \addcontentsline{toc}{subsubsection}{%
     #1}%
  }
\makeatother

\usepackage{chngcntr}
\counterwithin*{paragraph}{section}
\newcommand{\TPARA}[1]{\paragraph{\hspace{-0.4cm}\phantom{.}  \bfm{#1}} }

\makeatletter
\def\namedlabel#1#2{\begingroup
    #2%
    \def\@currentlabel{#2}%
    \phantomsection\label{#1}\endgroup
}
\makeatother

\makeatletter
\renewcommand*\env@matrix[1][*\c@MaxMatrixCols c]{%
  \hskip -\arraycolsep
  \let\@ifnextchar\new@ifnextchar
  \array{#1}}
\makeatother
\allowdisplaybreaks

\usepackage{verbatim}

\begin{document}

\title{\bscaps{An Asymmetric Random Rado Theorem: $1$-statement}}
\author{Elad Aigner-Horev} 
\address{Department of Mathematics and Computer Science, Ariel University, Ariel, Israel}
\email{horev@ariel.ac.il}

\author{Yury Person}
\address{Institut f\"ur Mathematik, Technische Universit\"at Ilmenau, 98684 Ilmenau, Germany}
\email{yury.person@tu-ilmenau.de}
\thanks{YP is supported by the Carl Zeiss Foundation.}

\date{\today}
\maketitle

\begin{abstract}
A classical result by Rado characterises the so-called partition-regular matrices $A$, i.e.\ those matrices $A$ for which any finite colouring of the positive integers yields a monochromatic solution to the equation $Ax=0$. We study the {\sl asymmetric} random Rado problem for the (binomial) random set $[n]_p$ in which one seeks to determine the threshold for the property that any $r$-colouring, $r \geq 2$, of the random set has a colour $i \in [r]$ admitting a solution for the matrical equation $A_i x = 0$, where $A_1,\ldots,A_r$ are predetermined partition-regular matrices pre-assigned to the colours involved.

We prove a $1$-statement for the asymmetric random Rado property. In the symmetric setting our result retrieves the $1$-statement of the {\sl symmetric} random Rado theorem established in a combination of results by R\"odl and Ruci\'nski~\cite{RR97} and by Friedgut, R\"odl and Schacht~\cite{FRS10}. We conjecture that our $1$-statement in fact unveils the threshold for the asymmetric random Rado property, yielding a counterpart to the so-called {\em Kohayakawa-Kreuter conjecture} concerning the threshold for the  asymmetric random Ramsey problem in graphs.

We deduce the aforementioned $1$-statement for the asymmetric random Rado property after establishing a broader result generalising the main theorem of Friedgut, R\"odl and Schacht from~\cite{FRS10}. The latter then serves as a combinatorial framework through which $1$-statements for Ramsey-type problems in random sets and (hyper)graphs alike can be established in the asymmetric setting following a relatively short combinatorial examination of certain hypergraphs. To establish this framework we utilise a recent approach put forth by Mousset, Nenadov and Samotij~\cite{MNS18} for the Kohayakawa-Kreuter conjecture. 
\end{abstract}

\section{Introduction}\label{sec:intro}

Locating the thresholds for various Ramsey properties of random structures has been of prime interest of late. After \L{}uczak, Ruci\'nski and Voigt~\cite{LRV91} launched the systematic study of these thresholds a great number of results followed. In a series of papers, R\"odl and Ruci\'nski~\cite{RR94,RR93,RR95} established a version for Ramsey's theorem\footnote{Ramsey's theorem asserts that for fixed positive integers $r$, $k$ and $\ell$ any $r$-colouring of $\binom{[n]}{k}$ with $n$ sufficiently large yields an $\ell$-element set $S\subset [n]$ with all sets from  $\binom{S}{k}$ being of the same colour.}~\cite{Ram30} in random graphs often referred to as  the {\em symmetric random Ramsey theorem}, where here the term 'symmetric' denotes that here the (hyper)graph sought to be found appearing monochromatically is the same across all colours; we use the term {\em asymmetric} if the configurations assigned to colours may differ. Ramsey properties of random hypergraphs were pursued in~\cite{CG16,FRS10,GNPSST17,NPSS17,NS16,RR98,RRS07}; asymmetric Ramsey properties of random graphs and hypergraphs were studied in~\cite{GNPSST17,KSS14,MSSS09,MNS18}. 

Ramsey theory also houses numerous problems seeking monochromatic configurations in the set of integers $[n]:=\{1,\ldots,n\}$; where here if to name a few one encounters for instance Schur's theorem\footnote{Schur's theorem asserts that in any finite colouring of $\NN$ there is always a monochromatic additive triple $(a,b,c)$ with $a+b=c$.}~\cite{Schur}; van der Waerden's theorem\footnote{van der Waerden's theorem asserts that any finite colouring of $\NN$ contains a monochromatic progression of any fixed length}~\cite{vdW27}. The reader can consult the book~\cite{GRS90} by Graham, Rothschild and Spencer for further such examples; in particular in what follows we devote much attention to a theorem by Rado that generalises the last two theorems. Such "Ramsey on the integers"-type problems were explored in the random setting as well~\cite{CG16,FRS10,GRR96,HST19,RR97,Sp17} and in fact for some of this problem sharp thresholds are known~\cite{FHPS16,FK00,FRRT06,SchSch18}.

An $\ell \times k$ matrix $A$ with integer entries is said to be {\em partition-regular} if any finite colouring of $\Naturals$ admits a monochromatic solution to the homogeneous matrical equation $A x = 0$. The matrical equation of Schur's theorem is the simple equation $x_1+x_2-x_3=0$; for van der Waerden's theorem  the system of linear equations consists of $x_1-2x_2+x_3=0$, $x_2-2x_3+x_4=0$,\ldots, $x_{k-2}-2x_{k-1}+x_k=0$. 
The characterisation of all partition-regular matrices is a classical result by Rado~\cite{Rado}, who showed that such matrices are captured through the so called  {\em columns condition} (see, e.g.,~\cite[Chapter~3]{GRS90} for details).  We would be remiss if we were not to remark that the matrix associated with van der Waerden's theorem is an example of what is commonly referred to as a {\em density-regular} system. 
  
A partition-regular matrix $A$ is said to be {\em irredundant} if the equation $Ax = 0$ has a solution $\begin{bmatrix}x_1\; \cdots \; x_k\end{bmatrix}^{\top}$ {\em non-repetitive} in the sense that $x_i \not= x_j$ for $1\leq i < j \leq k$; otherwise the matrix $A$ is said to be {\em redundant}. Every redundant matrix admits an $\ell' \times k'$ irredundant submatrix $A'$ with $\ell' < \ell$ and $k' < k$ such that the sets of solutions for the equations $Ax =0$ and $A'y=0$ are the same when viewed as sets (see e.g.,~\cite{FRS10,RR97} for details). Owing to this, one may restrict the discussion to irredundant partition-regular matrices, for which we may also assume full row rank. Consequently, we refer to irredundant partition-regular matrices of full row rank as {\em Rado matrices}. 

For a subset $X \subseteq [n]$, an integer $r \geq 1$, and a Rado matrix $A$, we write 
$
X \to (A)_r
$
in order to denote that every $r$-colouring of $X$ admits a monochromatic solution for the matrical equation $A x = 0$. The aforementioned result of Rado~\cite{Rado} coupled with a classical compactness argument (see, e.g.,~\cite{GRS90}) asserts that if $A$ is a Rado matrix then $[n] \to (A)_r$ for every sufficiently large $n$. 

For $p \in [0,1]$, let $[n]_p$ denote the binomial random subset of $[n]$ where members of $[n]$ are included independently at random each with probability $p$. Since  $X \to (A)_r$ is 
an increasing monotone property\footnote{If $X \to (A)_r$ then $Y \to (A)_r$ whenever $Y\supseteq X$.}, a {\sl threshold} for the property $[n]_p \to (A)_r$ exists by a result of Bollob\'as and Thomason~\cite{BT87}; i.e., there exists a function $\hat{p}\colon \NN\to[0,1]$ such that $\PP\left[[n]_p\to (A)_r\right]\longrightarrow 1$ whenever $p=\omega(\hat{p})$ (the \emph{$1$-statement}, hereafter), and such that $\PP\left[[n]_p\to (A)_r\right]\longrightarrow 0$ whenever for $p=o(\hat{p})$ (the \emph{$0$-statement}, hereafter). 

Graham, R\"odl and Ruci\'nski~\cite{GRR96} studied Schur's theorem for two colours in random sets and determined the threshold of this Ramsey property to be $n^{-1/2}$. R\"odl and Ruci\'nski~\cite{RR97} studied the Rado's theorem in random sets of integers; they determined the $0$-statement for the associated property and provided the $1$-statement for a special case of Rado matrices, namely the aforementioned {\sl density regular} matrices. 
The $1$-statement in its full generality was established later on by Friedgut, R\"odl and Schacht in~\cite{FRS10} and thus establishing the so called {\em symmetric random Rado theorem}. More recently, resilience versions of this problem were studied by Hancock, Staden and Treglown~\cite{HST19} and by Spiegel~\cite{Sp17}.

The following parameter introduced first in~\cite{RR97} arises in the threshold of the symmetric random Rado property. For an $\ell\times k$ Rado matrix $A$, set
\begin{equation}\label{eq:mA}
m(A) := \max_{\substack{W \discup \overline{W} = [k] \\ |W| \geq 2 }} \frac{|W|-1}
{|W|-1+\rk(A_{\overline{W}}) - \rk A},
\end{equation}
where here $\rk A$ denotes the rank of $A$ and for $I \subseteq [k]$ the term $A_I$ denotes the submatrix of $A$ obtained by restricting $A$ to the columns whose index lies in $I$. This parameter is well-defined~\cite{RR97}. 

\begin{theorem}\label{thm:FSR}{\em (Symmetric random Rado theorem~\cite[Theorem~3.1]{RR97},~\cite[Theorem~1.1]{FRS10})}\label{thm:random_Rado}\\
Let $A$ be a Rado matrix and let $r \in \Naturals$. There exist constants $0 < c < C$ such that the following holds
\[
\lim_{n \to \infty} \Pro \bigg[ [n]_{p(n)} \to (A)_r \bigg] =
\begin{cases}
0, & p(n) \leq cn^{-1/m(A)}\\
1, & p(n) \geq Cn^{-1/m(A)},
\end{cases}
\]
where $p\colon\Naturals\to[0,1]$.
\end{theorem}

Given $r \geq 2$ partition-regular matrices, namely $A_1,\ldots,A_r$, we write $X \to (A_1,\ldots,A_r)$ to denote that $X$ has the property that for any $r$-colouring of its elements there exists a colour $i \in [r]$ such that the matrical equation $A_i x = 0$ has a solution in colour $i$. In this case $X$ is said to have the {\em asymmetric} Rado property (w.r.t. the matrices $A_1,\ldots,A_r$). 

The asymmetric Rado property for $\Naturals$ and any $r$ Rado matrices can be deduced directly from the characterisation of partition-regular matrices due to Rado~\cite{Rado}. Indeed, given $A_1,\ldots,A_r$ Rado matrices, then the diagonal block matrix $B:=\mathrm{\mbs{diag}}(A_1,\ldots,A_r)$ is also partition-regular as it satisfies the columns condition of Rado~\cite{Rado} which is equivalent to partition-regularity. As such, in any finite colouring of $\Naturals$ the homogeneous matrical equation $B x =0$ has a monochromatic solution which can be further "decomposed" into $r$ monochromatic solutions for each equation $A_i y = 0$ -- this in fact exceeds the requirement in the asymmetric case. This observation does not yield however a good estimate on the threshold for the random set $[n]_p$ since the "density" $m(B)$ is much higher from what heuristics suggests. 

The only nontrivial result about the threshold for the random set $[n]_p\to (A_1,\ldots,A_r)$ in the literature is due to 
Hancock, Staden, and Treglown~\cite[Theorem~4.1]{HST19} who in fact considered the {\sl resilience} version of Theorem~\ref{thm:Rado} and were the first to study asymmetric Rado property in a random setting. 
They proved an upper bound of the form $Cn^{-1/m(A_1)}$ for the threshold, where $m(A_1)\ge m(A_i)$ for all $i\in[2,r]$. Again,  heuristics below suggests that this is far from the right threshold whenever $m(A_1)>m(A_2)\ge m(A_i)$ for $i\in[3,r]$. 

Our main result is the $1$-statement for what we conjecture to be the threshold for $[n]_p \to (A_1,\ldots,A_r)$. The following parameter arises in our result. 

\begin{definition}\label{def:mAB}
Let $A$ and $B$ be two Rado matrices, where $A$ is an $\ell_A \times k_A$-matrix and $B$ is an $\ell_B \times k_B$-matrix. Set
\[
m(A,B):=\max_{\substack{W\subseteq [k_A] \\ |W|\ge
2}}\frac{|W|}{|W|-\rk A+\rk(A_{\overline{W}}) -1+1/m(B)}.
\]
\end{definition}

\noindent
The parameter $m(A,B)$ is well-defined (see~\eqref{eq:luck}). 
Our main result reads as follows.

\begin{theorem}\label{thm:Rado}{\em (Main Result)}
Let $A_1, \ldots, A_r$ be $r$ Rado matrices satisfying 
$m(A_1) \geq m(A_2) \geq \cdots \geq m(A_r)$. Then there exists a constant $C >0$ such that
\[
\lim_{n \to \infty} \Pro \left[ [n]_{p} \to (A_1,\ldots,A_r) \right] = 1.
\]
whenever $p \ge  C n^{-1/m(A_1,A_2)}$.
\end{theorem}

A special case when the matrices $A_i$ describe arithmetic progressions (asymmetric van der Waerden theorem) was proved independently and simultaneously  by Zohar~\cite{Zohar}, see more information in the concluding remarks section, Section~\ref{sec:conclude}.

One can easily verify the  equality $m(A,A) = m(A)$ (see the proof of Observation~\ref{obs:const} in Section~\ref{sec:aux}), and therefore Theorem~\ref{thm:Rado} retrieves the $1$-statement of the symmetric random Rado theorem (see Theorem~\ref{thm:FSR}) when the matrices $A_1,\ldots,A_r$ coincide. We conjecture that the $n^{-1/m(A_1,A_2)}$ is in fact the threshold for the associated property $[n]_p\longrightarrow (A_1,\ldots, A_r)$. 

\begin{conjecture}\label{conj:AHP}
Let $A_1, \ldots, A_r$ be $r$ Rado matrices satisfying 
$m(A_1) \geq m(A_2) \geq \cdots \geq m(A_r)$. There exist constants $0 <c < C$ such that  the following holds
\[
\lim_{n \to \infty} \Pro \bigg[ [n]_{p} \to (A_1,\ldots,A_r) \bigg] =
\begin{cases}
1, & p \geq Cn^{-1/m(A_1,A_2)}\\
0, & p \leq cn^{-1/m(A_1,A_2)}.
\end{cases}
\]
\end{conjecture}

Our main result, namely Theorem~\ref{thm:Rado}, arises quite naturally in  arithmetic Ramsey problems for randomly {\sl perturbed} dense sets of integers. 
That is, for $n$ sufficiently large, given a set $N \subseteq [n]$ with positive density the distribution $N \cup [n]_p$ is viewed as a random perturbation of $N$. The limiting behaviour of the {\sl symmetric} Rado property $N \cup [n]_p \to (A)_2$ is then of interest where $A$ is a prescribed Rado matrix. The study of {\sl symmetric} Ramsey properties of randomly perturbed dense graphs, namely $G \cup G(n,p)$ with $G$ a dense graph, originates with the work of Krivelevich, Sudakov and Tetali~\cite{KST}. Recently much progress has been attained by Das and Treglown~\cite{DT19} for the case of graphs. The $1$-statement of the Kohayakawa-Kreuter conjecture arises fairly naturally in this type of results for graphs and we forgo the details here. For the integers, much is less known. The authors~\cite{AHP} have established that $p = n^{-2/3}$ is the threshold for the densely perturbed set $N \cup [n]_p$ to admit the Schur property; yet no other result in this venue is currently known for any other Rado matrix. This is partly due to an asymmetric random Rado type theorem at the correct threshold being missing from the literature; an issue we conjecture to be mended here.

Our proof of Theorem~\ref{thm:Rado} builds upon the ideas of Mousset, Nenadov and Samotij~\cite{MNS18}. We in fact deduce Theorem~\ref{thm:Rado} from a broader result (namely Theorem~\ref{thm:main}) that provides a general {\sl combinatorial framework} for deducing $1$-statements for asymmetric random Ramsey-type results in random (hyper)graphs and sets alike. Theorem~\ref{thm:main} generalises a result of Friedgut, R\"odl and Schacht from~\cite{FRS10} who provide such a combinatorial framework for $1$-statements of symmetric random Ramsey-type problems. Our proof of Theorem~\ref{thm:main} relies on the {\em container method}~\cite{BMS15,ST15} and the clever sparsification "trick" from~\cite{MNS18}. We postpone the statement of Theorem~\ref{thm:main} until the next section. Roughly speaking though, given an asymmetric Ramsey-type problem in random integer sets or (hyper)graphs involving configurations, say $C_1,\ldots,C_r$, for which one seeks a $1$-statement, Theorem~\ref{thm:main} calls for the examination of certain combinatorial properties of the {\sl solution hypergraphs} associated with each of the configurations $\{C_i\}_{i \in [r]}$. That is, for configuration $C_i$ the {\sl solution} hypergraph associated with $C_i$ is the one whose edge set consist of all "copies"/"solutions" (of) $C_i$ in the {\sl complete} universe (i.e., $K_n$ or $[n]$). Theorem~\ref{thm:main} asserts that if these $r$ solution hypergraphs satisfy a short list of combinatorial properties then the desired $1$-statement for the associated asymmetric random Ramsey-type problem would follow. We will  make this precise in \S~\ref{sec:technical_thm}.

The intuition underlying the parameter $m(A,B)$ and its involvement in our result is as follows. Let $A$ be an $\ell\times k$ Rado matrix and 
let $H^{(A)}$ denote the $k$-tuples $x\in[n]^k$ forming solutions to $Ax=0$ and  by $H^{(A)}_I$ the set of tuples {\sl projected} to $I$-indexed coordinates for some $I\subseteq [k]$. The common {\sl rule of thumb} sort of speak for the location of the threshold in the symmetric setting arises from requiring that the expected number of solutions to $Ax=0$ captured by $[n]_p$ dominates the expected size of $[n]_p$. Often, this is not enough and one in fact must require that the expected number of {\sl projected} solutions to $Ax=0$, i.e.\ the set of the form $\{x_I\colon Ax=0\}$ for \emph{any} $\emptyset\neq I\subseteq [k]$, captured by $[n]_p$ dominates the expected size of $[n]_p$.  
This requirement is embodied in the maximisation seen in~\eqref{eq:mA}. 

The parameter $m(A,B)$ arises in a similar manner. For a sequence $p$ sufficiently "small", say, one would like to colour $[n]_p$ as to avoid, say, a red solution to $Ax=0$ and, say, a blue solution to $B y=0$. With $p$ set, the (expected) density of the set of solutions to $Ax=0$ in $[n]_p$ is $q:=\Theta(\min_{\emptyset\neq I\subseteq[k_A]} |H^{(A)}_I|p^{|I|}/n)$. Moreover, at least one element from each of the solutions to $Ax=0$ in $[n]_p$ should be coloured blue. In fact this set of blue elements can be thought of as being randomly distributed in some sense.  But if the `expected' number (which is of the order  at least $\min_{\emptyset\neq I\subseteq[k_B]} |H^{(B)}_I|q^{|I|}$) of projected solutions  to $By=0$ captured by this (random) set  exceeds $qn$, then it "should" be impossible to avoid a blue solution to $By=0$, and here one observes the similarity with the symmetric case. This intuitive explanation is of course quite crude; nevertheless, the parameter $m(A,B)$ can be seen to emerge in this way. 

The reader familiar with asymmetric Ramsey properties in random (hyper)graphs will undoubtedly draw parallels between the so-called {\em Kohayakawa-Kreuter conjecture} (see Conjecture~\ref{conj:KK} below) and our Conjecture~\ref{conj:AHP}; more specifically one cannot help but to compare $m(A,B)$ to the graph parameter arising for the threshold  in the Kohayakawa-Kreuter conjecture. 
For indeed, the intuition underlying the location of `most' thresholds in the Kohayakawa-Kreuter conjecture is as described above for the asymmetric random Rado problem. 

Ramsey's theorem~\cite{Ram30} asserts that for sufficiently large $n$ any colouring of the edges of the complete $r$-uniform hypergraph $K^{({r})}_n=([n],\binom{[n]}{r})$ with $k$ colours admits a monochromatic copy of $F$; this is captured concisely with the notation $K^{({r})}_n\longrightarrow (F)_k$. This generalises to the asymmetric case as to read $H\longrightarrow (F_1,\ldots,F_k)$ in a straightforward manner. The binomial random hypergraph $H^{({r})}(n,p)$ is defined by choosing each of the $\binom{n}{r}$ possible edges independently at random with probability $p$. If $r=2$ then this is the binomial random graph model commonly denoted by $G(n,p)$. For a nonempty $r$-uniform hypergraph $F$ the $m_r$-{\em density} of $F$ is given by $m_r(F):=\max_{F'\subseteq F, v(F')>r} d_r(F')$, where $d_r(F)=\frac{e(F)-1}{v(F)-r}$ if its number of edges $e(F)>1$, and $d_r(F)=1/r$ if $e(F)=1$ and $v(F)=r$. 

For a fixed $k\ge 2$, $F$ an $r$-uniform hypergraph and $p\ge Cn^{-1/m_r(F)}$ (for some absolute constant $C>0$) it does indeed hold that $\PP\left[H^{({r})}(n,{p})\longrightarrow (F)_k\right]\longrightarrow 1$ as $n\to\infty$. In \emph{many} cases $n^{-1/m_r(F)}$ is known to be the  \emph{threshold}; nevertheless there are exceptions.  R\"odl and Ruci\'nski~\cite{RR94,RR93,RR95} determined for every graph $F$ and any fixed number of colours the thresholds for the random graph $G(n,p)$. The case of random hypergraphs is not fully solved, but a general $1$-statement was given by Friedgut, R\"odl and Schacht~\cite{FRS10} and by Conlon and Gowers~\cite{CG16}  (in the strictly balanced case), the matching $0$-statement for cliques was provided by Nenadov, Person, \v{S}kori\'c and Steger~\cite{NPSS17}.  Gugelmann, Nenadov, Person, \v{S}kori\'c, Steger and Thomas~\cite{GNPSST17} also discovered another type of behaviour in random hypergraphs which is not exhibited in the random graph $G(n,p)$. 
   
Amongst the first to consider asymmetric Ramsey properties in random graphs were Kohayakawa and Kreuter~\cite{KK97} who studied the case of cycles and put forth a conjecture as to where to locate the thresholds. For graphs $F_1$ and $F_2$ with $m_2(F_1)\ge m_2(F_2)$ the 
   asymmetric \emph{$m_2$-density} of $F_1$ and $F_2$ is given by 
\[
 m_2(F_1,F_2):=\max_{F_1'\subseteq F_1, e(F'_1)\ge 1}\frac{e(F'_1)}{v(F'_1)-2+1/m_2(F_2)}.   
\]
The Kohaykawa-Kreuter conjecture is then as follows where here we take the version from~\cite{KSS14}.

\begin{conjecture}\label{conj:KK} {\em (The Kohayakawa-Kreuter conjecture)}\\
Let $F_1$, \ldots, $F_r$ be graphs with $m_2(F_1)\ge m_2(F_2)\ge \ldots\ge m_2(F_r)$ and $m_2(F_2)>1$. Then 
 the threshold for $G(n,p)\longrightarrow (F_1,\ldots,F_r)$ is $n^{-1/m_2(F_1,F_2)}$.
\end{conjecture}

This conjecture has been studied in~\cite{KK97,KSS14,MSSS09}. Kohayakawa and Kreuter verified the conjecture for cycles, 
Marciniszyn, Skokan, Sp\"ohel and Steger~\cite{MSSS09} proved the $0$-statement for cliques and observed that the $1$-statement in the  strictly `balanced' case would follow from the so-called K\L{}R-conjecture (which was verified later in~\cite{BMS15,CGSS14,ST15}) and Kohayakawa, Schacht and Sp\"ohel~\cite{KSS14} proved the strictly-balanced case. The asymmetric hypergraph analogue was studied in~\cite{GNPSST17}, where the $1$-statement was proven for general graphs with an additional $\log n$-factor. 
In a recent paper Mousset, Nenadov and Samotij~\cite{MNS18} managed to erase this $\log n$-factor using a clever sparsification trick. The proof of~\cite{MNS18} (as well as of~\cite{GNPSST17}) makes use of the container method~\cite{BMS15,ST15}. 

Support for Conjecture~\ref{conj:AHP} can be found in the fact that our combinatorial framework for deducing $1$-statements for asymmetric random Ramsey-type results, namely Theorem~\ref{thm:main}, recovers the $1$-statements for graphs and hypergraphs at the threshold conjectured by the Kohayakawa-Kreuter conjecture and its extension to hypergraphs. We revisit this statement in the remarks following the statement of Theorem~\ref{thm:main}.

The organisation of the paper is as follows. In Section~\ref{sec:technical_thm} we describe our main technical result (Theorem~\ref{thm:main}). In Section~\ref{sec:MNS} we prove Theorem~\ref{thm:main} by following closely the approach of Mousset, Nenadov and Samotij~\cite{MNS18}. In Section~\ref{sec:aux} we provide combinatorial results about Rado matrices, in Section~\ref{sec:Rado} we deduce Theorem~\ref{thm:Rado} from Theorem~\ref{thm:main} and Section~\ref{sec:conclude} contains some concluding remarks.

\section{Main technical result}\label{sec:technical_thm}
The purpose of this section is to state our main technical result, namely Theorem~\ref{thm:main}, from which Theorem~\ref{thm:Rado} is deduced with relative ease. Theorem~\ref{thm:main} is proven using the container method~\cite{BMS15,ST15} and is stated in the spirit of the main result~\cite[Theorem~2.5]{FRS10}. 
 We begin with the statements of the combinatorial properties that the aforementioned solution hypergraphs are required to satisfy for Theorem~\ref{thm:main} to take effect. Roughly speaking these properties fall into four categories to which we refer to as: {\sl containerability}, {\sl Ramseyness}, {\sl tameness}, and {\sl boundedness}. In what follows we make these precise.
 
Throughout, a sequence of hypergraphs $\bH:=(H_n)_{n \in \Naturals}$ is said to have property $\P$ if $H_n$ has property $\P$ whenever $n$ is sufficiently large. We sometimes refer to $H_n$ as solution hypergraphs.
\vspace{1.5ex}

\noindent
\scaps{Containerability.} For some of the solution hypergraphs involved in the asymmetric random Ramsey-type problem we require that the container method can be applied to them. One can capture this using the following functions introduced in~\cite[Section~3.1]{ST15}.   

Let $H$ be a $k$-uniform $n$-vertex hypergraph with average vertex degree $d > 0$. For $2 \leq j \leq k$ and $v \in V(H)$ set
\[
\deg_H^{(j)}(v) : = \max \{\deg_H(T): v \in T \subseteq V(H) \mand |T|= j\}, 
\]
where $\deg_H(T)$ is the number of edges of $H$ that contain $T$.
For $\tau > 0$ one defines 
\[
\delta_j(H,\tau) := \frac{\sum_{v\in V(H)} \deg_H^{(j)}(v)}{\tau^{j-1}nd},
\]
and 
\[
\delta(H,\tau):= 2^{\binom{k}{2}-1}\sum_{j=2}^k 2^{-\binom{j-1}{2}}\delta_j(H,\tau),
\]
which is the \emph{co-degree function} from~\cite[Section~3.1]{ST15}.
\vspace{1.5ex}

\noindent
\scaps{Tameness.} Like containerability the next property also involves degrees of the associated solution hypergraphs. However, tameness can be seen to be more intimately related to the configuration at hand. For indeed this property calls for the {\sl extendability} of so-called {\sl sub-solutions} into complete solutions to be controllable to a certain extent.

An {\sl ordered} ($k$-uniform) hypergraph $\cH: = (H,\bpi)$ is a pair comprised of an $n$-vertex $k$-uniform hypergraph $H$ and a set of bijections $\bpi$, 
where each element $\pi\in\bpi$, which may be viewed as a $k$-tuple,  is a bijection from $[k]$ to some edge $e\in E(H)$ denoted $e_{\pi}$, i.e.\ $e_\pi=\pi([k])$. We thus view elements of $\bpi$ as \emph{ordered edges} of $\cH$. We also write $e$, $f\in\cH$ for such ordered edges, and notations $e\subseteq A$ or $e\cap f$ mean that we view $e$ and $f$ as sets $e([k_1])$ and $f([k_1])$ (by dropping the order).  
We write $|\cH|$  for $|\bpi|$, i.e.\ the number of ordered edges in $\cH$, and we also identify $\cH$ with $\bpi$.

To put this in context of, say, the Rado problem, the edges of such ordered hypergraphs $\cH$ will arise later in Section~\ref{sec:Rado} as solution vectors $x$ with distinct entries to the equations of the form $Ax=0$, where $A$ is some Rado matrix and $x\in [n]^k$. The r\^ole of the bijections $\pi_e$ from $\bpi$ is to record the positions of the elements of $e$ as these are to be placed into the solution vector $x$.

For $\emptyset \not= I \subseteq [k]$ and $\pi \in \bpi$, we write $\pi|_I$ for the restriction of $\pi$ to $I$. 
The $I$-{\em projection} of $\cH$ is defined to be 
\[
\cH_I: = \left(H_I,\bpi_I\right)\text{ where } H_I:=(V(H),\{\pi(I)\colon \pi\in \bpi\})\text{ and } \bpi_I:=\left\{\pi|_I: \pi\in\bpi \right\}.
\]
In particular $\cH_{[k]}$ coincides with $\cH$. Given $e$, $f\in\cH$ and $u\in\cH_I$, we write $e\cap f=u$ if $e|_I=u=f|_I$ and $e([k_1])\cap f([k_1])=u(I)$. 

Observe that for $I \subset W \subseteq [k]$ the $I$-projection of $\cH_W$ coincides with the $I$-projection of $\cH$, and that several edges of $\cH$ may indeed be projected onto the same edge of an $I$-projection. 
For $\emptyset \not= I\subseteq [k]$ and (an ordered edge) $y\in \cH_I$ write 
$
\deg_{\cH}(y) := |\{\pi\in \bpi:\pi|_I=y\}|
$
to denote the \emph{degree} of $y$ in $\cH=(H,\bpi)$.

The following definition captures a setting in which the degrees of projected edges are not much larger than the average.

\begin{definition}
Let $K \in \Naturals$. An ordered  hypergraph $\cH=(H,\bpi)$ is said to be $K$-{\em tamed} if 
for all $\emptyset \not= I \subset W \subseteq [k]$ 
\begin{equation}\label{eq:extend}
\deg_{\cH_W}(u) \leq K \frac{|\cH_W|}{|\cH_I|}
\end{equation}
holds for all $u \in\bpi_I$. 
\end{definition}

For future reference let us note that $K$-tamed hypergraphs $\cH$ have the property that whenever $\emptyset \not= I \subset W \subseteq [k]$ it holds that:
\begin{equation}\label{eq:cherry}
\sum_{u \in \cH_I} \deg_{\cH_W}(u)^2 \le K^2 \frac{|\cH_W|^2}{|\cH_I|}.
\end{equation}
In particular when $W = [k]$ the condition becomes 
\begin{equation}\label{eq:traditional-boundedness}
	\sum_{u \in \cH_I} \deg_{\cH}(u)^2 \le K^2 \frac{|\cH|^2}{|\cH_I|}.
\end{equation} 

Later (in Section~\ref{sec:Rado}), we will use $K$-tameness of an ordered $k$-uniform hypergraph $\cH=(H,\bpi)$ to get bounds 
on the co-degree function $\delta_j(H,\tau)$ by using $\deg^{(j)}_H (v) \le \sum_{I \in \binom{[k]}{j}} \Delta^{I}(\cH)$, where $\Delta^{I}(\cH)$ is the maximum over all 
$\deg_{\cH}(u)$ with $u\in\cH_I$. 

\vspace{1.5ex}

\noindent
\scaps{Ramseyness}. Another combinatorial property that we shall require is {\em Ramsey supersaturation} that is fit to the asymmetric setting.
Given (possibly ordered) hypergraphs $H_1$, \ldots, $H_r$, then $(H_i\colon i\in[r])$ is said to be $r$-{\em Ramsey} if for any vertex partition $U_1 \discup \cdots \discup U_r$ of $V(H)$ there exists an $i \in [r]$ such that $e(H_i[U_i]) > 0$. Observe that for $H_1=\ldots=H_r$ this reduces to the symmetric setting. 

We will be working with the following quantitative version of the Ramsey property.
Given $r \in \Naturals$ and $i \in [r]$, let $(H_n^{(i)})_{n\in \Naturals}$ be a sequence of $k_i$-uniform (possibly ordered) hypergraphs with the property that for every $n \in \Naturals$ the hypergraphs $(H_n^{(i)})_{i \in [r]}$ are all defined on the common vertex set $[n]$. 

Put $\fH_n:=\left((H_n^{(i)})\colon i \in [r]\right)$ and $\bfH:=\left((H_n^{(i)})\colon i \in [r]\right)_{n\in\Naturals}$. 
The sequence $\bfH$ is said to be $(r,\zeta)$-{\em Ramsey} if for every sufficiently large $n$ and for any vertex partition $U_1 \discup U_2 \discup \cdots \discup U_r$ of $[n]$ there exists an $i\in [r]$ such that $e(H^{(i)}_n[U_i]) > \zeta e(H^{(i)}_n)$. 
\vspace{1.5ex}

\noindent
\scaps{Boundedness.} Most technical of all properties is that of {\sl boundedness} and it is here that we encounter the sparsification trick of~\cite{MNS18}. Roughly speaking the property essentially calls for a {\sl weight} function to be put on elements of the random set and thus {\sl sparsifying} it as to have that after sparsification the expected number of $I$-projected solutions for every nontrivial subset of indicies $I$ be asymptotically comparable with the size of the containers employed through the containerability property.

\begin{definition}\label{def:bounded}{\em ($(p,w,\tau)$-boundedness)}
Let $2 \leq k \in \Naturals$, let $p\colon\Naturals\to(0,1]$, $w:\P\left([k]\right) \to [1,\infty)$ and $\tau := \tau(n)\colon \NN\to (0,1)$ be functions, where additionally $\tau n \to \infty$ with $n$.  
A sequence $(\cH_n)_{n \in \Naturals}$ of ordered $k$-uniform hypergraphs is said to be $(p,w,\tau)$-{\em bounded} if 
	\begin{equation}\label{eq:bounded}
	\min_{\emptyset \not= I \subseteq [k]} p^{w(I)}|(\cH_n)_I| = \Theta(\tau n)
	\end{equation}
holds for all sufficiently large $n$, i.e.\ there exist absolute constants $c',C'>0$  with 
\[
c'\tau n\le \min_{\emptyset \not= I \subseteq [k]} p^{w(I)}|(\cH_n)_I|\le C'\tau n
\]
for all large $n$.
\end{definition}

\noindent
In the context of~\eqref{eq:bounded}, a set $W \subseteq [k]$ satisfying
\[
p^{w(W)}|(\cH_n)_W| = \min_{\emptyset \not=  I \subseteq [k]} p^{w(I)}|(\cH_n)_I|
\]
is called a $w$-{\em minimiser set}. If the associated function $w$ has the property that for every $ i\in [k]$ there exists a $w$-minimiser set containing $i$ then $w$ is called {\em proper}. We write $\W(w)$ for the set of $w$-minimisers whenever $w$ is proper. A sequence of ordered hypergraphs that is $(p,w,\tau)$-bounded with $w$ being proper is said to be $(p,w,\tau)$-{\em properly bounded}.
\vspace{2ex}

We are now ready to state our main technical result. 
 In the context of the asymmetric random Rado problem, say, this result conveys the message that upon collating all solution hypergraphs for all matrical equations involved in the problem into a single (ordered not necessarily uniform) hypergraph, then if the latter satisfies the above four combinatorial properties (namely containerability, Ramseyness, tameness, and boundedness) then a.a.s. this hypergraph, once restricted to the elements/vertices chosen by the random set, will support the desired Rado property. While true in spirit (and certainly in the symmetric case), the asymmetric nature of the Ramsey-type properties we are after 
renders the above description slightly inaccurate in the sense that asymmetry will require the satisfaction of the above combinatorial properties in a manner not as 
{\sl homogeneous} as described. This we now make precise; to that end will prefer to write $V_{n,p}$ in order to denote the binomial random set $[n]_p$. 

\begin{theorem}\label{thm:main}{\em (Main technical result)}
Let $2 \leq r, k_1,\ldots, k_r \in \Naturals$,  and let $\zeta , K, \eps>0$ such that $\eps < 1/2$ and $\zeta > t! \eps$, where $t:= \max_{i \in [2,r]}k_i$. For $i \in [r]$, let $(\cH_n^{(i)})_{n\in \Naturals}$ be a sequence of $k_i$-uniform ordered hypergraphs such that the hypergraphs $\{H_n^{(1)},\ldots,H_n^{(r)}\}$ are all defined on the same vertex set $[n]$ for every $n \in \Naturals$.  There exists a $C>0$ such that the following holds for every $p\colon\Naturals\to (0,1]$ 
 satisfying $p(n)\longrightarrow 0$ as $n\to\infty$.  

If  $\bfH=\left((\cH_n^{(i)})\colon i \in [r]\right)_{n\in\Naturals}$ is $(r,\zeta)$-Ramsey, 
$\cH_n^{(1)}$ is $K$-tamed and $(p,w,\tau)$-properly bounded, and $\delta(H^{(i)}_n,\tau) \leq \eps/12 t!$ holds for all $i=2$,\ldots, $r$, then a.a.s.  $(\cH^{(i)}_n[ V_{n,q}]\colon i\in[r])$ is $r$-Ramsey whenever $q:= q(n) \geq C p(n)$. 
\end{theorem}

\begin{remark}
The condition $p(n)\longrightarrow 0$ is not very restrictive (and, in fact, it can be omitted at the expense of choosing $c'$ in~\eqref{eq:bounded} appropriately). Moreover, in applications one is concerned with sparse cases anyway. Since Ramsey properties are monotone, the truth of the statement for $p(n)\longrightarrow 0$ implies it for larger probabilities as well.
\end{remark}

\begin{remark} 
Our Theorem~\ref{thm:main} is stated in the spirit of the main result for symmetric Ramsey problems of Friedgut, R\"odl and Schacht~\cite[Theorem~2.5]{FRS10}. 
The boundedness condition there implicitly involves a form of the co-degree function combined with probabilities $p(n)$, whereas our theorem treats them separately.   As a consequence, the verification for the (ordered) hypergraphs arising as solutions to the linear equations involving Rado matrices are short and straightforward.
\end{remark}

\begin{remark}
In~\S~\ref{sec:intro} we claimed to have Theorem~\ref{thm:main} reproduce the $1$-statement seen in the Kohayakawa-Kreuter conjecture (see Conjecture~\ref{conj:KK}) which has been recently proved in its full generality by Mousset, Nenadov and Samotij~\cite{MNS18}. This while the condition  $m_2(F_2)>1$ stated in the Kohayakawa-Kreuter conjecture is missing from the premise of Theorem~\ref{thm:main}. As noted in~\cite{KSS14}, in the Kohayakawa-Kreuter conjecture the condition $m_2(F_2) > 1$ is required for the conjectured {\sl threshold} to hold; dropping this condition does not refute the $1$-statement; the latter remains true only not at the optimal density. 
\end{remark}

\section{A generalised Mousset-Nenadov-Samotij type argument}\label{sec:MNS}
In this section we prove Theorem~\ref{thm:main}. The proof is an adaptation of the arguments from~\cite{MNS18} and thus we follow~\cite{MNS18} closely throughout this section.  

Let $(\cH_n^{(i)})_{n\in \Naturals}$ and $\bfH=\left((\cH_n^{(i)})\colon i \in [r]\right)_{n\in\Naturals}$ be as in Theorem~\ref{thm:main}.  We will  be considering sequences of the form $(A_n,\xi_n)_{n \in \Naturals}$ where $A_n \subseteq [n]$ and $\xi_n : [n] \to [k_1]$ is a function viewed as a $k_1$-partition of $A_n=\dcup_{i=1}^{k_1}(\xi^{-1}(i)\cap A_n)$; consequently, we refer to $(A_n,\xi_n)$ as a $k_1$-{\em partite set}. Often we will suppress the index $n$ and treat $(A_n,\xi_n)$ as the pair 
$$
\A := (A \subseteq [n],\xi: [n] \to [k_1]).
$$ 
 Whenever $\xi$ is clear from the context, we identify $\A$ with $A$. 
A set $e \subseteq [n]$ with $|e| =k_1$ is said to be $\xi$-{\em partite} if all its members lie in different parts of $\xi$, i.e.\ $\xi(e)=[k_1]$.

For $i \in [r]$, we denote by $\cH^{(1)}_n [\A]$ the ordered subgraph of $\cH^{(1)}_n$ with edges $\pi$ satisfying  $\pi([k_1])\subseteq A_n$ and $\pi^{-1}=\xi|_{\pi([k_1])}$ (i.e.\ those edges $\pi$ which respect the partition $\xi$), and we refer to such edges as \emph{$\xi$-partite}. Generally, we also denote projections $\pi|_W$ of  $\pi$ to  $W$ as $\xi$-partite if $\pi^{-1}|_{\pi(W)}=\xi|_{\pi(W)}$. 
For $X \subseteq [n]$ we write $E_\xi(\cH^{(1)}_n[X])$ to denote the $\xi$-partite edges of $\cH^{(1)}_n$ spanned by $X$ and $e_\xi(H^{(1)}_n[X])$ to denote the cardinality of this set. We say that $\fH_n[\A]$ is $r$-{\em partite-Ramsey} if for every partition $A_1 \discup A_2\dcup \cdots \dcup A_r$ of $A$ either $e_\xi(\cH^{(1)}_n[A_1]) > 0$ or there exists an $i \in [2,r]$ such that $e(\cH^{(i)}_n[A_i]) > 0$.

Let us assume that $\fH_n[\A]$ is not $r$-partite-Ramsey. Then there exists an $r$-colouring $A_1 \discup A_2 \discup \cdots \discup A_r$ of $A$ such that $e_\xi (\cH^{(1)}_n[A_1]) = 0$ and $e(\cH^{(i)}_n[A_i]) = 0$ for every $i=2$,\ldots, $r$. This in turn implies that there exists an $r$-colouring $A'_1 \discup A'_2 \discup \cdots \discup A'_r$ of $A$ such that any $v \in [n]$ that does not lie in a $\xi$-partite edge of $\cH^{(1)}_n$ satisfies $v \in A'_1$.
 Since $e(\cH^{(i)}_n[A_i]) = 0$, the sets $A'_2$,\ldots, $A'_r$ are independent in the hypergraphs $\cH^{(2)}_n$,\ldots, $\cH^{({r})}_n$ and also in the hypergraphs $H^{(2)}_n$,\ldots, $H^{({r})}_n$ where the order of the vertices in edges is dropped. Moreover, every edge of a $k_i$-uniform $H^{(i)}_n$ gives rise to at most $k_i!$ ordered hyperedges in $\cH^{({i})}_n$, and hence the number of edges between $\cH^{({i})}_n[X]$ and $H^{({i})}_n[X]$ always differs by a constant factor.

The independent sets above can be 
 approximately described  by the following  version of the container theorem due to Saxton and Thomason~\cite{ST15}. The set $\P A$ below denotes the power set of a set $A$,  and $\P(A)^s$ denotes the $s$-fold Cartesian product of $\P A$.  

\begin{theorem}\label{thm:containers} {\em~\cite[Corollary~3.6]{ST15}} Let $H$ be a $k$-uniform hypergraph on $[n]$. Let $0 < \eps,\tau < 1/2$ and let $\tau$ satisfy $\delta(H,\tau) \leq \eps/12k!$.  
Then there exists a constant $c := c(k)=800k!^3k$ and a function $f\colon \P([n])^s\to\P[n]$ where $s\le c\log(1/\eps)$, with the following properties. Let 
$\cT=\{(T_1,\ldots, T_s)\in\cP([n])^s\colon |T_i|\le c\tau n, 1\le i\le s\}$, and let $\cC=\{f(T)\colon T\in \cT\}$. Then
\begin{enumerate}
	\item  [\namedlabel{itm:C1}{(C.1)}] for every independent set $I$ there exists a {\em signature} $T := (T_1,\ldots,T_s) \in \cT\cap\P(I)^s$ such that 
	$
	I \subseteq f(T) \in \C, 
	$
	\item  [\namedlabel{itm:C2}{(C.2)}] $e(H[C]) \leq \eps e(H)$ for all $C \in \C$, 
	\item  [\namedlabel{itm:C3}{(C.3)}] $\log|\C|\le c\log (1/\eps)n\tau\log(1/\tau)$.
\end{enumerate}
\end{theorem}

We further call the sets from $\C$ \emph{containers}. 

Given $\zeta,\eps, K, \tau$, and $p$ (per the quantification of Theorem~\ref{thm:main}) let $c_i := c_i(k_i)$ and $s_i := s_i(\eps,c_i)$ be the constants guaranteed by Theorem~\ref{thm:containers}, as this will be applied to every member of $(H^{(i)}_n)_{i=2,\ldots, r}$ with $\eps$ and $\tau$. Set  
\begin{equation}\label{eq:const}
c := \max_{i \geq 2} c_i \mand s := \max_{i\geq 2} s_i.
\end{equation}
In addition let $(f_i)_{i\geq 2}$ denote the mappings from signatures to containers as defined in~\ref{itm:C1} for each such application. 

By Theorem~\ref{thm:containers} applied with $\eps$ and $\tau$ to $H^{(i)}_n$, $i \geq 2$, there exists a collection of containers $\C_i \subseteq \P[n]$  such that for each $A'_i$, $i\geq 2$, there exists a signature $T^{(i)} := (T^{(i)}_1,\ldots,T^{(i)}_{s_i}) \in \P(A'_i)^{s_i}$ (per~\ref{itm:C1}).
Given the signatures $(T^{(2)},\ldots,T^{(r)})$ note that $A \setminus (f_2(T^{(2)}) \cup \cdots \cup f_r(T^{(r)})) \subseteq A'_1$ and thus 
$$
e_\xi(\cH^{(1)}_n \big[ A \setminus (f_2(T^{(2)}) \cup \cdots \cup f_r(T^{(r)}))\big])=e_\xi(\cH^{(1)}_n \big[ A'_1\big]) = 0.
$$

The following observation summarises the above discussion. 

\begin{observation}\label{obs:witnesses}
Let $\A : = (A,\xi)$ be a $k_1$-partite set. If $\fH_n[\A]$ is not $r$-partite-Ramsey then, there is a partition of 
$A=A'_1\dcup A'_2\dcup \ldots\dcup A'_r$ (as described above), where the following two properties are met:
\begin{enumerate}
\item [\namedlabel{itm:P1}{(P.1)}] there exists a sequence of signatures $
	(T^{(2)},\ldots,T^{(r)})$ as defined above such that 
	$$
	\bigcup_{i=2}^r \bigcup_{j = 1}^{s_i} T^{(i)}_j \subseteq A,
	$$
and every member of $\bigcup_{i=2}^r\bigcup_{j = 1}^{s_i} T^{(i)}_j$ lies in a $\xi$-partite edge of $\cH^{(1)}_n$. By assumption $(\cH^{(1)}_n)_{n\in\Naturals}$ is $(p,w,\tau)$-properly bounded; thus, for any $v \in \bigcup_{i=2}^r \bigcup_{j = 1}^{s_i} T^{(i)}_j$ an arbitrary $\xi$-partite edge of $\cH^{(1)}_n[A]$ containing $v$ can be fixed, and thus a $w$-minimiser set $W_v \subseteq [k_1]$ satisfying $\xi(v) \in W_v$ can be assigned to $v$. The set $Z_v:=\xi^{-1}(W_v)$ (containing $v$) is viewed as a \emph{witness} for $v$. Define
 \begin{equation}\label{eq:witness-collection}
\tA:=\bigcup_{v\in \bigcup_{i=2}^r \bigcup_{j = 1}^{s_i} T^{(i)}_j } Z_v,
 \end{equation}
and note that 
\begin{equation}\label{eq:tilde-A-bound}
|\tA|\le k_1\cdot r\cdot s\cdot c\cdot \tau n
\end{equation} 
holds.
\item [\namedlabel{itm:P2}{(P.2)}] We have $e_\xi(\cH^{(1)}_n \big[ A \setminus (f_2(T^{(2)}) \cup \cdots \cup f_r(T^{(r)}))\big])=0$.
\end{enumerate}
\end{observation}

Next we define a weighted partite random set, which generalises the model $[n]_p$.

\begin{definition} \label{def:Vnpwk}
Given $n,k \in \Naturals$, $p \in [0,1]$, and $w : [k] \to [1,\infty)$,  the probability distribution $V_{n,p,w,k}$ on $[n]$ is defined as follows. 
\begin{enumerate}
	\item Choose a function $\xi:[n] \to [k]$  uniformly at random. 
	\item An element $x \in [n]$ is included independently in $V_{n,p,w,k}$ with probability $p^{w(\xi(x))}$. 
\end{enumerate}
\end{definition}

From now on we write $V_{n,p,w}$ to denote $V_{n,p,w,k_1}$, where $k_1$ is the uniformity of the hypergraphs in $(\cH^{(1)}_n)_{n\in\Naturals}$. 
We consider two probabilities namely $p\colon\Naturals\to[0,1]$ and $q\colon\Naturals\to[0,1]$, and we 
shall write $p$ and $q$ instead of $p(n)$ and $q(n)$, respectively.  As $V_{n,q,w}$ and  $V_{n,q}$ can be coupled so that  $V_{n,q,w} \subseteq V_{n,q}$ 
the following holds

\begin{equation}\label{eq:reason}
\Pro[\fH_n[V_{n,q}] \; \text{is not $r$-Ramsey}] \leq \Pro [\fH_n[V_{n,q,w}] \; \text{is not $r$-partite-Ramsey}].
\end{equation}

Before proceeding further, we introduce the following quantity $X_I:= |\{\bi \in (\cH^{(1)}_n)_I: \bi \subseteq V_{n,q,w} \}|$ and prove the following fact about it.

\begin{lemma}\label{lem:concentration}
For every $\emptyset \not= I \subseteq [k_1]$, 
$$
\lim_{n\to \infty}\Pro \left[X_I \leq 2 q^{w(I)}|(\cH^{(1)}_n)_I|\right] = 1.
$$
\end{lemma}

\begin{proof}
We start by noting that $\Ex X_I  = q^{w(I)}|(\cH^{(1)}_n)_I|$. Moreover, we may write $X_I = \sum_{f \in (\cH^{(1)}_n)_I} \ind_{f}$ where $\ind_{
f}$ is the indicator variable for whether $f \in (\cH^{(1)}_n)_I$ satisfies $f \subseteq V_{n,q,w}$. By Chebyshev's inequality 
\begin{equation}\label{eq:cheby}
\Pro \left[X_I > 2 q^{w(I)}|(\cH^{(1)}_n)_I| \right] \leq \frac{1}{\Omega(\Ex X_I)} + \frac{\Delta}{\Omega((\Ex X_I)^2)}
\end{equation}
where 
$$
\Delta := \sum_{\substack{f,f' \in (\cH^{(1)}_n)_I \\ f \not= f' \\ f \cap f' \not = \emptyset }}\Ex [\ind_{f} \cdot \ind_{f'}].
$$
For the latter quantity we have 
\begin{align*}
\Delta &\overset{\phantom{\eqref{eq:cherry}}}{\leq} \sum_{\emptyset \not= J \subset I} \sum_{\substack{f,f' \in (\cH^{(1)}_n)_I \\ \xi(f\, \cap\, f') = J}} \Ex [\ind_{f} \cdot \ind_{f'}]\\ 
& \overset{\phantom{\eqref{eq:cherry}}}{\leq} \sum_{\emptyset \not=J} q^{2w(I)-w(J)} \sum_{u \in (\cH^{(1)}_n)_J} \deg_{(\cH^{(1)}_n)_I}(u)^2\\
& \overset{\eqref{eq:cherry}}{\leq} \sum_{\emptyset \not=J} q^{2w(I)-w(J)} K^2 \frac{|(\cH^{(1)}_n)_I|^2}{|(\cH^{(1)}_n)_J|}\\
& \overset{\phantom{\eqref{eq:cherry}}}{=} K^2 q^{2w(I)}|(\cH^{(1)}_n)_I|^2 \sum_{\emptyset \not=J} \left(q^{w(J)}|(\cH^{(1)}_n)_J| \right)^{-1}.
\end{align*}
Substituting this estimate for $\Delta$ in~\eqref{eq:cheby} one arrives at 
\begin{align*}
\Pro \left[X_I > 2 q^{w(I)}|(\cH^{(1)}_n)_I| \right] & \leq \frac{1}{\Omega(q^{w(I)}|(\cH^{(1)}_n)_I|)} + \frac{K^2 q^{2w(I)}|(\cH^{(1)}_n)_I|^2 \sum_{\emptyset \not=J} \left(q^{w(J)}|(\cH^{(1)}_n)_J| \right)^{-1}}{\Omega(q^{2w(I)}|(\cH^{(1)}_n)_I|^2)}\\
& \leq \frac{1}{\Omega(q^{w(I)}|(\cH^{(1)}_n)_I|)} + \sum_{\emptyset \not=J} \frac{O(1)}{q^{w(J)}|(\cH^{(1)}_n)_J|}.
\end{align*}
Both summands in the last estimate vanish owing to~\eqref{eq:bounded} since it guarantees  that $q^{w(X)}|(\cH^{(1)}_n)_X| \to \infty$ for every $\emptyset\neq X \subseteq [k_1]$.
\end{proof}

The role of the set $\A$ from Observation~\ref{obs:witnesses} will be assumed by the set $V_{n,q,w}$, hence $A=V_{n,q,w}$ and $\xi$ is the random function from $V_{n,q,w}$. 
Since the set $\tA\subseteq A$,  we will assume from now on that the set $\tA$ is such that $|(\cH^{(1)}_n[\tA])_I|\le 2 q^{w(I)}|(\cH^{(1)}_n)_I|$ for all $\emptyset \not= I \subseteq [k_1]$. This indeed holds with probability $1-o(1)$, by the Lemma above, and  this fact will be exploited towards the end of the proof.

Now we exploit our  Observation~\ref{obs:witnesses}  as follows:
\begin{align*}
\Pro [\fH_n[V_{n,q,w}] & \; \text{is not $r$-partite-Ramsey}] \\ 
& \le \sum_{(T^{(2)},\ldots,T^{(r)})}\sum_{\tA} \PP[e_\xi(\cH^{(1)}_n \big[ V_{n,q,w}\setminus (f_2(T^{(2)}) \cup \cdots \cup f_r(T^{(r)})) \big]) = 0\text{ and }\tA\subseteq V_{n,q,w}]\\
& \le \sum_{(T^{(2)},\ldots,T^{(r)})}\sum_{\tA} \PP[e_\xi(\cH^{(1)}_n \big[ V_{n,q,w}\setminus (f_2(T^{(2)}) \cup \cdots \cup f_r(T^{(r)})\cup\tA) \big]) = 0\text{ and }\tA\subseteq V_{n,q,w}]\\
&= \sum_{(T^{(2)},\ldots,T^{(r)})}\sum_{\tA} \PP[e_\xi(\cH^{(1)}_n \big[ V_{n,q,w}\setminus (f_2(T^{(2)}) \cup \cdots \cup f_r(T^{(r)})\cup\tA) \big]) = 0]\cdot\PP[\tA\subseteq V_{n,q,w}],
\end{align*}
where the sums are over all possible signatures $(T^{(2)},\ldots,T^{(r)})$ and the sets $\tA$ associated with them, which may arise as described in~\ref{itm:P1}, and $\xi$ is the random partition of $[n]$ from the definition of  $V_{n,q,w}$. 
Therefore, the remainder of the section will be concerned with establishing
\begin{equation}
\sum_{(T^{(2)},\ldots,T^{(r)})}\sum_{\tA} \PP[e_\xi(\cH^{(1)}_n \big[ V_{n,q,w}\setminus (f_2(T^{(2)}) \cup \cdots \cup f_r(T^{(r)})\cup\tA) \big]) = 0]\cdot\PP[\tA\subseteq V_{n,q,w}]
 = o(1),\label{eq:main}
\end{equation}
which would prove Theorem~\ref{thm:main}. 

\begin{lemma}\label{lem:Janson1}
For  any choice of $\tA$ per~\eqref{eq:witness-collection}
\begin{equation}\label{eq:Janson-arg}
\Pro \Bigg[ e_\xi(\cH^{(1)}_n \big[ V_{n,q,w}\setminus (f_2(T^{(2)}) \cup \cdots \cup f_r(T^{(r)})\cup\tA) \big]) = 0\Bigg] \leq 2^{-\Omega\left(\min_{\emptyset\neq I \subseteq [k_1]} q^{w(I)}\left|\left(\cH^{(1)}_n\right)_I\right|\right)}
\end{equation}
holds.
\end{lemma}

\begin{proof} Let  $(T^{(2)},\ldots,T^{(r)})$ be the signatures associated with $\tA$ (as specified in~\ref{itm:P1}). 
Owing to~\ref{itm:C2}, the $r$-partition of $[n]$ given by 
$$
[n] \setminus (f_2(T^{(2)}) \cup \cdots \cup f_r(T^{(r)})), f_2(T^{(2)}), \ldots ,f_r(T^{(r)}),
$$
has the property that $e(H^{(i)}_n[f_i(T^{(i)})]) \leq \eps e(H^{(i)}_n)$ and hence $e(\cH^{(i)}_n[f_i(T^{(i)})]) \leq \eps k_i! e(\cH^{(i)}_n)$  for every $i=2$,\ldots, $r$. 
As  $\fH_n$ is $(r,\zeta)$-Ramsey and $\eps t! < \zeta$, we obtain
$$
e(\cH^{(1)}_n \big[ [n] \setminus (f_2(T^{(2)}) \cup \cdots \cup f_r(T^{(r)}))\big]) \geq \zeta e(\cH^{(1)}_n). 
$$
The assumption that $\cH^{(1)}_n$ is $K$-tamed implies that each member of $\tilde A$ lies in at most $K \frac{|\cH^{(1)}_n|}{\min_{i\in [k_1]} \left|(\cH^{(1)}_n)_{\{i\}}\right|}$ edges of $\cH^{(1)}_n$. Owing to~\eqref{eq:bounded}, upheld by $\cH^{(1)}_n$ by assumption, 
$$
\min_{\emptyset \not= I \subseteq [k_1]} \left|(\cH^{(1)}_n)_I\right| = \Omega(\tau n / p )
$$
holds; where here we used the fact that $w(I) \geq 1$ for every $\emptyset \not= I \subseteq [k_1]$. Recalling that $|\tA|\le k_1\cdot r\cdot s\cdot c\cdot \tau n$, by~\eqref{eq:tilde-A-bound}, it follows that at most
$$
k_1\cdot r\cdot s\cdot c\cdot \tau n \cdot \frac{|\cH^{(1)}_n|}{\Omega(\tau n/p)} \leq (\zeta /2) |\cH^{(1)}_n|   
$$
edges of $\cH^{(1)}_n$ involve $\tA$; where here we use the fact that $p \to 0$. We may then write that
\[
e(\cH^{(1)}_n \big[[n] \setminus (f_2(T^{(2)}) \cup \cdots \cup f_r(T^{(r)}) \cup \tA)\big])\ge \zeta e(\cH^{(1)}_n)-|\tA|\cdot K \frac{|\cH^{(1)}_n|}{\min_{i\in [k_1]} \left|(\cH^{(1)}_n)_{\{i\}}\right|}\ge 
(\zeta/2) e(\cH^{(1)}_n).
\]

For an edge $e \in E(\cH^{(1)}_n)$, we have
$$
\Pro \left[ e \in  E_\xi(\cH^{(1)}_n[V_{n,q,w}]) \right] \geq q^{w([k_1])}/k_1^{k_1}
$$
where $w([k_1]) := \sum_{i \in [k_1]} w(i)$ and the term $k_1^{k_1}$ is incurred by the need for the edge to be $k_1$-partite. Consequently,
$$
\mu:= \Ex \left[ e_\xi(\cH^{(1)}_n \big[ V_{n,q,w}  \setminus (f_2(T^{(2)}) \cup \cdots \cup f_r(T^{(r)}) \cup \tA)\big]) \right] = \Omega_{\zeta,k_1}(1)q^{w([k_1])} e(\cH^{(1)}_n). 
$$
Gearing up towards an application of Suen's inequality (see below), set
$$
\ind_{e} := 
\begin{cases}
1, & e \in E_\xi(\cH^{(1)}_n[V_{n,q,w}]),\\
0, & \text{otherwise}
\end{cases}
$$
and consider the quantities:  
$$
\Delta:=\frac{1}{2}\sum_{\substack{e,f \in E(\cH^{(1)}_n)\\ e\, \cap\, f \not= \emptyset}} \Ex [\ind_{e} \ind_{f}]
\mand 
\delta:=\max_{e\in E(\cH^{(1)}_n)}\sum_{\substack{e,f \in E(\cH^{(1)}_n)\\ e\, \cap\, f \not= \emptyset}} \Ex [\ind_{f}]
$$ 
estimations of which are required for the subsequent application of Suen's inequality.

For $\delta$, the following  upper bound 
$$
\delta\le \max_{\emptyset\neq I\subseteq [k_1]} k_1!\cdot k_1\cdot K q^{w([k_1])}\frac{|\cH^{(1)}_n|}{\left|\left(\cH^{(1)}_n\right)_I\right|}
$$ 
holds; where here we relied on~\eqref{eq:extend}. 
Next, for the correlation term $2\Delta$ we have that 
\begin{align*}
2\Delta & \overset{\phantom{\eqref{eq:traditional-boundedness}}}{:=} \sum_{\substack{e,f \in E(\cH^{(1)}_n)\\ e\, \cap\, f \not= \emptyset}} \Ex [\ind_{e} \ind_{f}] \\
& \overset{\phantom{\eqref{eq:traditional-boundedness}}}{\leq} \sum_{\emptyset\neq I \subseteq [k_1]} \sum_{\substack{u \in \left(\cH^{(1)}_n\right)_I} }\sum_{\substack{e,f \in E(\cH^{(1)}_n) \\ e \, \cap \, f = u}} \Ex[\ind_{e} \ind_{f}] \\
& \overset{\phantom{\eqref{eq:traditional-boundedness}}}{\leq} \sum_{I\neq \emptyset} \sum_{u} \sum_{e,f} q^{2w([k_1])-w(I)} \\ 
& \overset{\phantom{\eqref{eq:traditional-boundedness}}}{\leq} \sum_{I\neq \emptyset} q^{2w([k_1])-w(I)} \sum_{u \in \left(\cH^{(1)}_n\right)_I} \deg_{\cH^{(1)}_n}(u)^2\\
& \overset{\eqref{eq:traditional-boundedness}}{\leq} \sum_{I\neq \emptyset} q^{2w([k_1])-w(I)} K^2 \frac{e(\cH^{(1)}_n)^2}{\left|\left(\cH^{(1)}_n\right)_I\right|} \\
& \overset{\phantom{\eqref{eq:traditional-boundedness}}}{=} O \left( \frac{\mu^2}{\min_{\emptyset\neq I \subseteq [k_1]}q^{w(I)}|\left(\cH^{(1)}_n\right)_I|} \right).
\end{align*}
The claim now follows by Janson's version~\cite{J98} of Suen's inequality:
\[
\PP\left[e_\xi(\cH^{(1)}_n \big[V_{n,q,w}\setminus (f_2(T^{(2)}) \cup \cdots \cup f_r(T^{(r)})\cup\tA)\big])=0 \right]\le \exp\left(-\min\left(\frac{\mu^2}{8\Delta},\frac{\mu}{2},\frac{\mu}{6\delta}\right)\right),
\]
and the estimates on $\mu$, $\Delta$, and $\delta$.
\end{proof}

Equipped with Lemma~\ref{lem:Janson1} we return to~\eqref{eq:main} and upon the appropriate substitution attain   
\begin{align}
\sum_{(T^{(2)},\ldots,T^{(r)})}\sum_{\tA} \PP[e_\xi(\cH^{(1)}_n \big[ V_{n,q,w}\setminus (f_2(T^{(2)}) \cup \cdots \cup f_r(T^{(r)})\cup\tA) \big]) = 0]\cdot\PP[\tA\subseteq V_{n,q,w}]\nonumber\\
\le 2^{-\Omega\left(\min_{\emptyset\neq I \subseteq [k_1]} q^{w(I)}\cH^{(1)}_n[I]\right)}\sum_{(T^{(2)},\ldots,T^{(r)})}\sum_{\tA}\Pro\left[\tA\subseteq V_{n,q,w}\right].\label{eq:centre2} 
\end{align}
 Recall, that  every set  of the form $\tA$ has the property 
 that each of its elements $v$ is covered by some witness set $Z_{u} \in \binom{[n]}{\le k_1}$ (for some element $u$, not necessarily $v$ itself) that is a subset of some partite edge
$e \in E_\xi(\cH^{(1)}_n)$ such that $\xi(Z_{u})$ is a $w$-minimiser (cf.~\ref{itm:P1} of Observation~\ref{obs:witnesses}). 
Since each such $\tA$ arises as the union over all witnesses $Z_{u}$, where $u$ is an element in some of the 
possible tuples of signatures $(T_2,\ldots,T_r)$, we can make the following definition:
\[
\fU_k := \{\tA \colon \text{$\tA$ is $k$-coverable}\},
\]  
where $\tA$ is said to be $k$-{\em coverable} if the least number of sets $Z_v$ required to form  $\tA$ as their union is $k$. By~\ref{itm:P1} each such set $\tA\in\fU_k$ can be covered using at most $r\cdot s\cdot c\cdot \tau n$  sets of the form $Z_v$. 
Consequently, the double sum appearing on the r.h.s. of~\eqref{eq:centre2}
can be estimated as
\begin{equation}\label{eq:centre3}
\sum_{(T^{(2)},\ldots,T^{(r)})}\sum_{\tA}\Pro\left[\tA\subseteq V_{n,q,w}\right] \leq 
(rs)^{k_1\cdot r\cdot s \cdot c \cdot \tau n}\sum_{k=1}^{ r\cdot s \cdot c \cdot \tau n} \sum_{\tA \in \fU_k} \Pro [\tA \subseteq V_{n,q,w}],
\end{equation}
where here the factor $(rs)^{k_1\cdot r\cdot s \cdot c \cdot \tau n}$ accounts for the number of possibilities to reconstruct the signature ensemble $(T^{(2)},\ldots,T^{(r)})$ from a given $\tA$. 
The minimality of $k$ involved in the $k$-coverability of a set $\tA \in \fU_k$ implies  that every set $\tA \in \fU_k$ gives rise to at most $k$ members in $\fU_{k-1}$ which can be attained by simply discarding precisely one of the sets of the form $Z_v$ involved in building $\tA$. That is, there are at most $k$ distinct members $\tA' \in \fU_{k-1}$ such that $\tA = \tA' \cup Z_v$ for some $v\in \tA$. 

Peering closer into this union we write $\tA = \tA' \cup \bi \cup \br$ as to distinguish between the {\sl intersection} $\bi = Z_v \cap \tA'$ and the {\sl remainder} of this set namely $\br$. With this in mind let us recall that $\W:= \W(w)$ was defined to be the set of proper $w$-minimisers and write 
\begin{equation}\label{eq:break}
k \sum_{\tA \in \fU_k} \Pro [\tA \subseteq V_{n,q,w}] \leq 
 \sum_{\tA' \in \fU_{k-1}} \sum_{W \in \W} \sum_{I \subseteq W} \sum_{\substack{\bi \subseteq \tA' \\ \bi \in (\cH^{(1)}_n)_I}} \sum_{\substack{\br \in (\cH^{(1)}_n)_{W \sm I} \\ \bi \cup \br \in (\cH^{(1)}_n)_W}} \Pro [\tA' \cup \br \subseteq V_{n,q,w}].
\end{equation}
 The sums seen on the right hand side of~\eqref{eq:break} are as follows. We consider the generation of all members in $\fU_k$ through the members of $\fU_{k-1}$ via unions of the latter with all possible sets of the form $Z_v$. Given $\tA' \in \fU_{k-1}$ we seek to traverse sets of the form $Z_v$ which extend $\tA'$. As each such set $Z_v$  is associated with $w$-minimiser (through $\xi$), the second sum goes over all possible options for $\xi(Z_v)$. The set $\bi \cup \br$ being this set $Z_v$ is required to be $\xi$-partite and such that $\xi(\bi \cup \br) = W$ (for $W \in  \W$ chosen in the second sum). The third sum ranges over all possible partite representations allowed for $\bi$ to assume. The fourth sum ranges over all subsets of $\tA'$ that may assume the r\^ole of $\bi$. Finally, the fifth sum ranges over all possible completions $\br$. We remind the reader that the notation $\bi \subseteq \tA'$ and $\br \cup \tA'$ means that we may view $\bi$, $\br$ resp., also as sets (by forgetting the order), and that we denote by $\bi \cup \br$  an ordered tuple according to $\xi$.

The events $\{\tA' \subseteq V_{n,q,w}\}$ and $\{\br \subseteq V_{n,q,w} \}$ are independent on account of $\tA'\subset \tA$ and $\br$ being disjoint. Then 
\begin{align}
k \sum_{\tA \in \fU_k} & \Pro [\tA \subseteq V_{n,q,w}] \leq \nonumber \\
& \sum_{\tA' \in \fU_{k-1}} \Pro [\tA' \subseteq V_{n,q,w}]\sum_{W \in \W} \sum_{I \subseteq W} \sum_{\substack{\bi \subseteq \tA' \\ \bi \in (\cH^{(1)}_n)_I}} \sum_{\substack{\br \in (\cH^{(1)}_n)_{W \sm I} \\ \bi \cup \br \in (\cH^{(1)}_n)_W}} \Pro [\br \subseteq V_{n,q,w}] \leq \nonumber\\
& \sum_{\tA' \in \fU_{k-1}} \Pro [\tA' \subseteq V_{n,q,w}]\sum_{W \in \W} \sum_{I \subseteq W}|(\cH^{(1)}_n[\tA])_I| q^{w(W \sm I)} \Delta^{I}((\cH^{(1)}_n)_W) \overset{\eqref{eq:extend}}{\leq} \nonumber \\
&  \sum_{\tA' \in \fU_{k-1}} \Pro [\tA' \subseteq V_{n,q,w}]\sum_{W \in \W} \sum_{I \subseteq W} |(\cH^{(1)}_n[\tA])_I| \, q^{w(W \sm I)} K \frac{|(\cH^{(1)}_n)_W|}{|(\cH^{(1)}_n)_I|};
\label{eq:upper}
\end{align}

An application of Lemma~\ref{lem:concentration} allows us to further estimate~\eqref{eq:upper} by appealing that $|(\cH^{(1)}_n[\tA])_I|\le X_I\le 2 q^{w(I)}|(\cH^{(1)}_n)_I|$ holds with high probability ($1-o(1)$):
\begin{align}
k \sum_{\tA \in \fU_k} \Pro [\tA \subseteq V_{n,q,w}] & \leq \sum_{\tA' \in \fU_{k-1}} \Pro [\tA' \subseteq V_{n,q,w}]\sum_{W \in \W} \sum_{I \subseteq W} 2q^{w(I)}|(\cH^{(1)}_n)_I| q^{w(W \sm I)} K \frac{|(\cH^{(1)}_n)_W|}{|(\cH^{(1)}_n)_I|} \nonumber \\
&  = \sum_{\tA' \in \fU_{k-1}} \Pro [\tA' \subseteq V_{n,q,w}]\sum_{W \in \W} \sum_{I \subseteq W}2 K q^{w(W)}|(\cH^{(1)}_n)_W| \nonumber\\
& \leq \sum_{\tA' \in \fU_{k-1}} \Pro [\tA' \subseteq V_{n,q,w}]\sum_{W \in \W} (2K)^{2|W|} q^{w(W)}|(\cH^{(1)}_n)_W| \label{eq:middle}
\end{align}
Noting that 
$$
\Ex |\fU_k| = \sum_{\tA \in \fU_k} \Pro [\tA \subseteq V_{n,q,w}] \mand \Ex |\fU_{k-1}| = \sum_{\tA' \in \fU_{k-1}} \Pro [\tA' \subseteq V_{n,q,w}]
$$
we may rewrite~\eqref{eq:middle} as
$$
\Ex |\fU_k| \leq \frac{\Ex |\fU_{k-1}| \sum_{W \in \W} (2K)^{2|W|} q^{w(W)}|(\cH^{(1)}_n)_W|}{k}.
$$
Owing to $|\fU_0| = 1$ we may write 
$$
\Ex |\fU_k| \leq \frac{\left(\sum_{W \in \W} (2K)^{2|W|} q^{w(W)}|(\cH^{(1)}_n)_W|\right)^k}{k!}.
$$
This combined with~\eqref{eq:centre3} and~\eqref{eq:break} now yields
$$
\sum_{(T^{(2)},\ldots,T^{(r)})}\sum_{\tA}\Pro\left[\tA\subseteq V_{n,q,w}\right] \leq (rs)^{k_1\cdot r\cdot s \cdot c \cdot \tau n}\sum_{k=1}^{r\cdot s \cdot c \cdot \tau n}  \frac{\left(\sum_{W \in \W} (2K)^{2|W|} q^{w(W)}|(\cH^{(1)}_n)_W|\right)^k}{k!}.
$$
Substituting this into~\eqref{eq:centre2} and using $k!\ge (k/e)^k$ we arrive at
\begin{multline}
\Pro [\fH_n[V_{n,q,w}] \; \text{is not $r$-partite-Ramsey}]\\ 
\overset{\phantom{w \geq 1}}{\leq}  
(rs)^{k_1\cdot r\cdot s \cdot c \cdot \tau n}2^{-\Omega\left(\min_{\emptyset\neq I \subseteq [k_1]} q^{w(I)}|(\cH^{(1)}_n)_I|\right)}  \sum_{k=1}^{r\cdot s \cdot c \cdot \tau n}  \left(\frac{\sum_{W \in \W} e(2K)^{2|W|} q^{w(W)}|(\cH^{(1)}_n)_W|}{k}\right)^k\label{eq:final_R}
\end{multline}
The function $x \mapsto (d/x)^x$ is increasing for $0 < x \leq d/e$. Therefore, the expression in the inner sum is maximised for $k=M$, where 
\[
M:=\min\left\{\sum_{W \in \W} (2K)^{2|W|} q^{w(W)}|(\cH^{(1)}_n)_W|,r\cdot s \cdot c \cdot \tau n\right\}.
\] 
In what follows we replace $q$ with $Cp$ (since the Ramsey property is monotone and $q\ge Cp$) and we also use 
$\min_{\emptyset \not= I \subseteq [k]} p^{w(I)}|(\cH^{(1)}_n)_I| = \Theta(\tau n)$), which leads us to 
\[
\left(\frac{\sum_{W \in \W} e(2K)^{2|W|} q^{w(W)}|(\cH^{(1)}_n)_W|}{M}\right)^M\le \max\left\{e^M, e^{O(\tau n\log C)}\right\}\le e^{O(\tau n\log C)},
\]
where we used $M\le r\cdot s \cdot c \cdot \tau n$ and we hide $s$, $r$, $c$ in the $O(\cdot)$-notation. From this we can bound the right hand side of~\eqref{eq:final_R} 
from above by:
\[
(rs)^{k_1\cdot r\cdot s \cdot c \cdot \tau n}2^{-\Omega\left(\min_{\emptyset\neq I \subseteq [k_1]} q^{w(I)}|(\cH^{(1)}_n)_I|\right)} r\cdot s \cdot c \cdot \tau n\cdot e^{O(\tau n\log C)}.
\]
Appealing to $(p,w,\tau)$-boundedness and settig again $q=Cp$, we may write for $C$ sufficiently large:
\begin{equation*}
 \Pro  [\fH_n[V_{n,q,w}] \; \text{is not $r$-partite-Ramsey}]
\le (rs)^{k_1\cdot r\cdot s \cdot c \cdot \tau n}2^{-C \cdot  \Omega (\tau n)} e^{O(\log(\tau n) + (\log C) \tau n))}=o(1),
\end{equation*}
where we exploted that  $\tau n \to \infty$ due to $(p,w,\tau)$-boundedness.  This completes the proof of Theorem~\ref{thm:main}.

\section{Auxiliary results for partition-regular matrices}\label{sec:aux}

Here we collect some facts about bounds on the number of solutions to the matrical equation $A x = b$, where $A$ is some $\ell\times k$  matrix with integer entries, $x\in\NN^\ell$ and $b\in\NN^k$. We write $\rk A$ to denote the rank of $A$. Further define $\overline{I} : = [k] \sm I$ whenever $I \subseteq [k]$. Recall that $A_I$ denotes the submatrix of $A$, where we only keep columns indexed by $i\in I$. For a given $\ell\times k$-matrix $A$, we write $\cH$ for the solution vectors $x\in[n]^k$ to the equation $Ax=0$. Given  $I\subseteq [k]$, we write  $\cHHI$ for the set of all projections $x_I$, where $x\in[n]^k$ is a solution to $Ax=0$. For $j\in[k]$ we write  $A_j$ to denote the $j$th column of $A$. For $J\subseteq [k]$, let $V(A_J)$ denote the vector space spanned by the columns of $A_J$. 
Given two sets $X$ and $Y$ we write $X^Y$ for the set of functions of the form $Y \to X$. For a set $N\subset \ZZ$, we write 
$A_{J}\cdot N^{J}:=\{\sum_{j\in J} \alpha_j A_{j}\colon \alpha_j\in N\}$. 

\begin{lemma}\label{lem:project}
Let $A$ be an $\ell\times k$ matrix with integer entries and with $\rk A=\ell$. Then there exists a constant $K=K_A>0$ so that
for every $I \subseteq [k]$ we have
\begin{equation}\label{eq:project}
|\cHHI| \leq K n^{|I| - \rk A + \rk A_{\overline{I}}}.
\end{equation}
\end{lemma}
\begin{proof}
Set $C:=V(A_{I})\cap V(A_{\overline{I}})$ and observe that 
$\rk A+\dim C=\rk A_{I}+\rk A_{\overline{I}}$ holds. For $y\in\cHHI$ there is an $x\in[n]^k$ with $Ax=0$ and $x_I=y$. Since $A_Iy=-A_{\overline{I}}x_{\overline{I}}$ 
we infer that $A_Iy\in C$. Let $b\in C\cap A_{\overline{I}}\cdot [n]^{\overline{I}}$. Next we estimate the number of solutions $y\in[n]^k$ with $A_Iy=-b$. 
 Let $I_1\subseteq I$ be such that $\rk A_{I_1}=\rk A_I$, hence for any choice of $z\in[n]^{I\setminus I_1}$, there is at most one solution to $A_Iy=-b$ with $y_{I\setminus I_1}=z$ (because $A_{I_1}y_{I_1}=-b-A_{I\setminus I_1}y_{I\setminus I_1}=-b-A_{I\setminus I_1}z$ has at most one solution due to the linear independence of the columns of $A_{I_1}$). 
 Thus, for each $b\in C\cap A_{\overline{I}}\cdot [n]^{\overline{I}}$, there are at most $n^{|I|-|I_1|}$ vectors $y\in\cHHI$ with $A_Iy=-b$. 
 
 Since every solution $x\in[n]^k$ to $Ax=0$ must satisfy $A_{\overline{I}}x_{\overline{I}}\in C$, it remains to estimate $|C\cap A_{\overline{I}}\cdot [n]^{\overline{I}}|$. 
 For every $i\in\ovI$, let $A'_i$ be the orthogonal projection of the column $A_i$ to $C$, i.e.\ $(A_i-A'_i)^T z=0$ for all $z\in C$, and let $A'$ denote the matrix, whose columns are othogonal projections of the columns of $A$ to $C$.  
 If $b\in C\cap A_{\ovI}\cdot [n]^{\ovI}$, then $b$ is a linear integer combination of $A'_i$ with $i\in\ovI$. Let $J\subseteq \ovI$ with $|J|=\dim C$ be such that $A'_j$ with $ j\in J$ form a basis for $C$. Every other $A'_i$ ($i\in \ovI$) is a rational linear combination of $\{A'_j\colon  j\in J\}$ where the coefficients only depend on the entries of the matrix $A$, i.e.\ $A'_i=\sum_{j\in J}A'_j\beta_{ij}$ with $\beta_{ij}=\frac{c}{d}$ with $c$, $d\in \ZZ$ and $|c|$, $|d|\le K'$ for some absolute constant $K'=K'_A$. It follows that $A'_{\ovI}\cdot[n]^{\ovI}\subseteq \{\sum_{j\in J} A'_j (\alpha_j+\sum_{i\in \ovI\setminus J}\alpha_i\beta_{ij})\colon \alpha_i\in [n]\}$, and it is not difficult to see that the number of possible coefficients for every $A'_j$ is at most $2 |\ovI| (K'!)^2 n=O(n)$. It follows that there exists a constant $K=K_A$ (we can take $K$ to be at most $(2 |\ovI| (K'!)^2)^{\dim C}$) with 
 \[
 |\cHHI|\le K n^{|I|-|I_1|+\dim C}= Kn^{|I|-\rk A+\rk A_{\ovI}},
 \]
 where we used $\rk A+\dim C=\rk A_{I}+\rk A_{\overline{I}}$.
\end{proof}

For an $\ell\times k$-matrix $A$, a subset $I\subseteq [k]$ and a vector $y\in [n]^I$, the 
\emph{degree} of $y$ in $H$, i.e., $\deg_\cH(y)$, is given by the number of $z \in [n]^{\overline{I}}$ such that 
\[
A_I y + A_{\overline{I}} z = 0.
\]
Similarly, for $I \subseteq W \subseteq [k]$ and $y\in[n]^I$, we write $\deg_{\cHHW}(y)$ for the number of $y'\in[n]^W$ with $y=y'_I$.

\begin{lemma}\label{lem:deg}
Let $A$ be an $\ell\times k$ matrix with $\rk A=\ell$. 
Let $\emptyset \neq I \subseteq W \subseteq [k]$. Then there exists a constant $K=K_A>0$ so that
\begin{equation}\label{eq:deg}
\deg_{\cHHW}(y) \leq K n^{|W \sm I| - \rk A_{_{\overline{I}}} + \rk A_{_{\overline{W}}}}
\end{equation}
holds for every $y \in \cHHI$. 
\end{lemma}

\begin{proof}
For a given $y\in \cHHI$, we need to estimate the number of projections $x_W$, where $x\in [n]^k$ is a  solution 
to $A_{\ovI}x_{\ovI}=-A_Iy$ and $x_I=y$. Since for two solution vectors $x'$, $x''$ with $x'_I=y=x''_I$ we have $A_{\ovI} (x'_{\ovI}-x''_{\ovI})=0$, we  will instead be interested in estimating 
the number of $z_{W\setminus I}$ so that the vectors $z$ are  solutions to  $A_{\ovI} z=0$ with $z\in[-n+1,n-1]^{\ovI}$ (as this would be an upper bound for $\deg_{\cHHW}(y)$). 
A straightforward adaptation of Lemma~\ref{lem:project} above yields an upper bound of the form 
\[
K n^{|W\setminus I| - \rk A_{\ovI} + \rk A_{([k]\setminus I)\setminus(W\setminus I)}}=
K n^{|W\setminus I| - \rk A_{\ovI} + \rk A_{\overline{W}}}.
\]
\end{proof}

For $\emptyset \not= I \subset W \subseteq [k]$, set $\Delta^{(I)}(\cHHW):= \max \{\deg_{\cH[W]}(\bi) : \bi \in \cHHI\}$.

\begin{lemma}\label{lem:lower}
Let $A$ be an $\ell\times k$ matrix with $\rk A=\ell$. 
If the matrical equation $A x = 0$ has $\Omega(n^{k - \rk A})$ solutions over $[n]^k$ then 
\[
|\cHHI| = \Omega\left(n^{|I|-\rk A + \rk A_{_{\overline{I}}}}\right)
\]
holds for every $I\subseteq [k]$. 
\end{lemma}

\begin{proof}
For suppose that $|\cHHI| = o\left(n^{|I|-\rk A + \rk A_{_{\overline{I}}}}\right)$ for one such $I \subseteq [k]$, then this assumption together with  Lemma~\ref{lem:deg} (applied to $W=[k]$, thus $\rk A_{\overline{[k]}} = 0$)  yield
\[
\Omega(n^{k-\rk A}) = \left|\cH\right| \le |\cHHI|\Delta^{(I)}(\cH) = o (n^{|I| - \rk A + \rk A_{\overline{I}}} \cdot n^{|[k]\sm I| - \rk A_{\overline{I}} + \rk A_{\overline{[k]}}}) = o(n^{k-\rk A}),
\]
a contradiction. 
\end{proof}

\noindent

Recall that a matrix $A$ is \emph{partition-regular} if in any finite coloring of $\NN$ there is a monochromatic solution to $Ax=0$. Frankl, Graham and R\"odl~\cite{FGR88} 
proved the following supersaturation properties of partition-regular matrices. 

\begin{theorem}\label{thm:super} {\em~\cite[Theorem~1]{FGR88}}
Let $A$ be a partition-regular $\ell \times k$ matrix of rank $\ell$ and let $r \in \Naturals$. There exists a $c:=c(r,A)$ such that for any $r$-colouring of $[n]$ with $n$ sufficiently large there exists a colour $i$ in which there are $\geq c n^{k - \ell}$ solutions all coloured $i$. 
\end{theorem}

This Ramsey supersaturation result, Lemma~\ref{lem:lower}, and Lemma~\ref{lem:project} render the following. 

\begin{corollary}\label{cor:project}
Let $A$ be an $\ell \times k$ partition-regular matrix of rank $\ell$. Then for every $I \subseteq [k]$ 
\begin{equation}\label{eq:project}
|\cHHI|= \Theta \left(n^{|I| - \rk A + \rk A_{\overline{I}}}\right).
\end{equation}
\end{corollary}

Finally we will be using two further properties of partition-regular matrices, which we collect in the following lemma, see, e.g.,~\cite[Proposition~4.3]{HST19}. 

\begin{lemma}{\cite[Proposition~4.3]{HST19}}
Let $A$ be an $\ell \times k$ irredundant partition-regular matrix and let $I \subseteq [k]$. 
\begin{enumerate}
	\item if $|I| = 1$ then 
	\begin{equation}\label{eq:depend}
	\rk A - \rk A_{\overline{I}} = 0.
	\end{equation}
	
	\item If $|I| \geq 2$ then 
	\begin{equation}\label{eq:luck}
	k-|I| - \rk A_{\overline{I}} \leq k- \rk A-1 - \frac{|I|-1}{m(A)}.
	\end{equation}
    \item 
    	\begin{equation}\label{eq:bound_mA}
	    m(A)>1
	\end{equation}
\end{enumerate}
\end{lemma}

We conclude this section with the following observation regarding the parameter $m(A,B)$ (see Definition~\ref{def:mAB}.  

\begin{observation}\label{obs:const}
Let $A$ and $B$ be two matrices of dimensions $\ell_A \times k_A$ and $\ell_B \times k_B$, respectively. If $m(A) \geq m(B)$ then $m(A,B) \geq m(B)$. 
In particular, $m(A,A)=m(A)$. 
\end{observation}

\begin{proof}
Let $U \subseteq [k_A]$ with $|U|\ge 2$ be the set defining $m(A)$. It suffices to show that 
\[
\frac{|U|}{|U|-\rk A + \rk A_{\overline{U}} - 1 + 1/m(B)} \geq m(B),
\]
since $\frac{|U|}{|U|-\rk A + \rk A_{\overline{U}} - 1 + 1/m(B)}$ is a lower bound on $m(A,B)$. 
We rewrite the inequality above as 
\[
|U| \geq m(B) (|U|-\rk A + \rk A_{\overline{U}} - 1) +1
\]
and then again as 
\[
\frac{|U|-1}{|U|-\rk A + \rk A_{\overline{U}} - 1} \geq m(B).
\]
Noticing that the l.h.s.\ of the last inequality equals $m(A)$ (by our choice of $U$), 
we  conclude the proof of this observation because $m(A) \geq m(B)$ holds by the initial assumption. 
\end{proof}

\section{Proof of Theorem~\ref{thm:Rado}}\label{sec:Rado}
In this section we deduce Theorem~\ref{thm:Rado} from Theorem~\ref{thm:main}. To that end let $A_1,\ldots,A_r$, $r \geq 1$, be Rado matrices such that 
$
m(A_1) \geq m(A_2) \geq \cdots \geq m(A_r)
$
with $A_i$ having dimensions $\ell_i \times k_i$. 
For $n \in \Naturals$ and $i \in [r]$ define $\cH^{(i)}_n=(H^{(i)}_n,\bpi^{(i)}_n)$ to be the ordered $k_i$-uniform hypergraph whose vertex set is $[n]$ and whose edge set is comprised of all solutions over $[n]$ for the matrical equation $A_i x = 0$ with pairwise distinct entries of the vectors $x$.  The sequences $(\cH^{(i)}_n)_{n \in \Naturals}$ are thus defined as well as the sequence $\bfH=(\fH_n)_{n \in \Naturals}$. Set $p:=p(n):= n^{-1/m(A_1,A_2)}$. 
We  seek to apply Theorem~\ref{thm:main} to the sequences $\bfH$ and $(\cH^{(i)}_n)_{n \in \Naturals}$. Hence we need to verify that these sequences satisfy the premise of Theorem~\ref{thm:main}.
\vspace{2ex}

\TPARA{Ramseyness}
The existence of $\zeta >0$ such that $\fH_n$ is $(r,\zeta)$-Ramsey whenever $n$ is sufficiently large is asserted by~\cite[Lemma~4.4]{HST19} who deduce this from the removal lemma seen at~\cite[Theorem~2]{KSV12}. Somewhat simpler one may set $B:= \mathrm{\mbs{diag}}(A_1,\ldots,A_r)$ which as noted in the Introduction is partition-regular. Then Theorem~\ref{thm:super} yields a constant $c(r,B)$ such that for any $n$ sufficiently large, any $r$-colouring of $[n]$ admits at least $c(r,B) n^{\sum_i k_i - \rk(B)} = c(r,B) n^{\sum_i (k_i -  \rk A_i)}$ monochromatic solutions to the matrical equation $B x = 0$. It follows that $\bfH$ is $(r,c(r,B)/r)$-Ramsey.  

 \TPARA{Tameness of $\cH^{(1)}_n$} 
Fix $\emptyset \not= I \subset W \subseteq [k_1]$. By Corollary~\ref{cor:project}
$$
\frac{\left|\left(\cH^{(1)}_n\right)_W\right|}{\left|\left(\cH^{(1)}_n\right)_I\right|} =\Theta \left( \frac{n^{|W|-\rk A + \rk A_{\overline{W}}}}{n^{|I|-\rk A+\rk A_{\overline{I}}}}\right) = \Theta \left(n^{|W \sm I| -\rk A_{\overline{I}} + \rk A_{\overline{W}}}\right).
$$
By~\eqref{eq:deg} 
$$
\deg_{\cH^{(1)}_n}(y) \leq K n^{|W \sm I| - \rk A_{\overline{I}} + \rk A_{\overline{W}}}
$$
holds for every $y \in \left(\cH^{(1)}_n\right)_I$. 
It follows that $\cH^{(1)}_n$ is $O(1)$-tamed. 
\vspace{2ex}

\TPARA{Containerability for $(H^{(i)}_n)_{i=2}^r$} 
We check that the conditions for `containerability' specified in~\cite[Corollary~3.6]{ST15} (see Theorem~\ref{thm:containers}) are met by $H^{(2)}_n$,\ldots, $H^{(r)}_n$. Pick $\eps < \min\{1/2, c(r,B)/r)$ and set $\tau := C'n^{-1/m(A_2)}$ where $C'$ is some sufficiently large constant. As $\tau = o(1)$, we need to verify for every $i \in [2,r]$ that a sufficiently large $C'$ implies (with our choice of $\tau$):
\begin{equation}\label{eq:tau}
\delta(H^{(i)}_n,\tau) \leq \eps/ 12 (k_i)!.
\end{equation}
To see~\eqref{eq:tau}, fix $i \in [2,r]$ and fix $j \in [2,k_i]$. 
Then (with $\overline{I}=[k_i]\setminus I$):
\begin{align*}
\sum_{v \in [n]} \deg^{(j)}_{H^{(i)}_n}(v) 
& \overset{\phantom{\eqref{eq:luck}}}{\leq} \sum_{v \in [n]}\sum_{I \in \binom{[k_i]}{j}} \Delta^{(I)}(\cH^{(i)}_n)\\
& \overset{\eqref{eq:deg}}{\leq} \sum_{v \in [n]} \sum_{I \in \binom{[k_i]}{j}} K n^{k_i - j - \rk (A_i)_{\overline{I}} + \rk (A_i)_{\overline{[k_i]}}} \\
& \overset{\phantom{\eqref{eq:luck}}}{=} \sum_{v \in [n]} \sum_{I \in \binom{[k_i]}{j}} K n^{k_i - j - \rk (A_i)_{\overline{I}}}\\
& \overset{\eqref{eq:luck}}{\leq} \sum_{v \in [n]} \sum_{I \in \binom{[k_i]}{j}} K n^{k_i - \rk A_i - 1 - \frac{j-1}{m(A_i)}} \\
& \overset{\eqref{eq:project}}{=} O_{k_i} \left( n\cdot  n^{-\frac{j-1}{m(A_i)}} \cdot \frac{e(H^{(i)}_n)}{n}\right).
\end{align*}
Then 
\begin{align*}
\delta_j(H^{(i)}_n,\tau) & \leq \frac{\sum_{v \in [n]} \deg^{(j)}(v)}{\tau^{j-1} \cdot n \cdot \left(\frac{k_i e(H^{(i)}_n)}{n}\right)} \\
& =  \frac{O_{k_i} \left( n^{-\frac{j-1}{m(A_i)}} e(H^{(i)}_n)\right)}{(C')^{j-1} \cdot k_i\cdot n^{-\frac{j-1}{m(A_2)}}\cdot e(H^{(i)}_n)} \\ 
& = \frac{O_{k_i}(n^{-1/m(A_i)})}{(C')^{j-1} \cdot k_i\cdot n^{-1/m(A_2)}};
\intertext{owing to $m(A_2) \geq m(A_i)$ for all $i \in [2,n]$ it follows that $n^{-1/m(A_i)} \leq n^{-1/m(A_2)}$ and thus}
\delta_j(H^{(i)}_n,\tau) & \leq \frac{O_{k_i}(1)}{(C')^{j-1}}.
\end{align*}
Then 
$$
\delta(H^{(i)}_n,\tau) = 2^{\binom{k_i}{2} - 1} \sum_{j=2}^{k_i}2^{-\binom{j-1}{2}} \delta_j(H^{(i)}_n,\tau) \leq 2^{\binom{k_i}{2} - 1} \sum_{j=2}^{k_i}2^{-\binom{j-1}{2}}\frac{O_{k_i}(1)}{(C')^{j-1}};
$$
from which the existence of a choice of $C'$ yielding~\eqref{eq:tau} is clear.

\TPARA{Boundedness of $\cH^{(1)}_n$} 
First we observe that there exists a function $w:[k_1] \to [1,\infty)$ such that for every $x \in [k_1]$ the following is true
\begin{equation}\label{eq:w}
\min \{ |I| - \rk A_1 + \rk (A_1)_{\overline{I}} -  w(I)/m(A_1,A_2): I \subseteq [k_1], x \in I\} = 1 - 1/m(A_2).
\end{equation}
 A similar statement concerning asymmetric graph densities was proven in~\cite[Lemma~8]{MNS18}. 
The proof of~\eqref{eq:w} is almost verbatim as the proof of~\cite[Lemma~8]{MNS18} (which follows by a compactness argument), 
but we provide the details for completeness here.
To see~\eqref{eq:w}, define
$$
r_x(w) : = \min\{|I| - \rk A_1 + \rk (A_1)_{\overline{I}}  - w(I)/m(A_1,A_2): I \subseteq [k_1], x \in I\} - 1 + 1/m(A_2)
$$
whenever $w:[k_1] \to [1,\infty)$ and $x \in [k_1]$. 

It suffices to prove that there exists a function $w$ such that $r_x(w) = 0$ for every $x \in [k_1]$. To that end set
$$
\F := \{w : [k_1] \to [1,\infty): r_x(w) \geq 0, \;\; \forall x \in [k_1]\}. 
$$
Viewed as a subset of $\Reals^{{k_1}}$, the set $\F$ is non-empty and bounded and thus compact. The fact that $\F$ is bounded is simple: it follows from 
$0\le r_x(w)\le \frac{1}{m(A_2)}-\frac{w(x)}{m(A_1,A_2)}$. We focus on the non-emptiness of $\F$. We argue that 
\begin{equation}\label{eq:1-is-in}
w \equiv 1 \in \F.
\end{equation}
To see~\eqref{eq:1-is-in}, recall first that by~\eqref{eq:depend} $\rk A_1 - \rk A_{1_{\overline{I}}} = 0$ whenever $I \subseteq [k_1]$ satisfies $|I| = 1$. Consequently, for $w \equiv 1$ and $I=\{x\}$ we get
$$
1-\rk A_1 + \rk (A_1)_{\overline{I}} -1/m(A_1,A_2)-1+1/m(A_2)\ge 1/m(A_2)-1/m(A_1,A_2)\ge 0,
$$
which follows from $m(A_1) \ge m(A_2)$ and Observation~\ref{obs:const}. If $|I|\ge 2$ and $I\ni x$ then we have from the definition of $m(A_1,A_2)$ that 
$w(I)/m(A_1,A_2)=|I|/m(A_1,A_2)\le |I|-\rk A_1+rk(A_1)_{\overline{I}}-1+1/m(A_2)$, 
and hence
\[
|I| - \rk A_1 + \rk (A_1)_{\overline{I}}  - w(I)/m(A_1,A_2) - 1 + 1/m(A_2)\ge 0.
\]
Thus, in any case we have for all $x\in[k_1]$ that $r_x(1)\ge 0$. This then concludes the proof of~\eqref{eq:1-is-in} so that $\F$ is non-empty.

The function $w \mapsto w([k])$ is continuous over the now known to be compact $\F$. Let $w^*$ be the maximum attained by the function $w \mapsto w([k])$ over $\F$. The function $w^*$ has the property $r_x(w^*) = 0$ for every $x \in [k_1]$. Otherwise, let $x'\in[k_1]$ such that $r_{x'}(w^*)>0$ then set $\tilde{w}(x):=w^*(x)+\eps\cdot 1_{x=x'}$ for a sufficiently small $\eps>0$. This clearly yields a contradiction to the maximality of $w^*$.

 Now we proceed with the verification of the boundedness of $\cH^{(1)}_n$. 
Recall that $\tau := C'n^{-1/m(A_2)}$ so that $\tau n \to \infty$ with $n$,  by~\eqref{eq:bound_mA}. We prove that $\cH^{(1)}_n$ is $(p,w,\tau)$-bounded. This amounts to establishing 
$$
\min_{I \subseteq [k_1]}p^{w(I)}\left|\left(\cH^{(1)}_n\right)_I\right| = \Theta(\tau n)
$$
for all sufficiently large $n$. Owing to Corollary~\ref{cor:project} and since $p=n^{-1/m(A_1,A_2)}$ and $\tau=C'n^{-1/m(A_2)}$ this has the form 
$$
\Theta\left(\min_{\emptyset\neq I \subseteq [k_1]}n^{|I|-\rk A_1 + \rk (A_1)_{\overline{I}}-\frac{w(I)}{m(A_1,A_2)}}\right) = \Theta\left(n^{\min_{\emptyset\neq I \subseteq [k_1]}\left(|I|-\rk A_1 + \rk (A_1)_{\overline{I}}-\frac{w(I)}{m(A_1,A_2)}\right)}\right) = \Theta(n^{1-1/m(A_2)})
$$
where the last equality is owing to~\eqref{eq:w}, i.e., the "definition" of $w$.

\vspace{3ex}
This concludes the proof of Theorem~\ref{thm:Rado}.

\section{Concluding remarks}\label{sec:conclude}
While finalising the writing of this manuscript we were made aware of the work of Zohar~\cite{Zohar} who established a so-called {\sl asymmetric} random van der Waerden theorem as follows. Given integers $\ell_1 \ge ... \ge \ell_r \ge 3$ there exist constants $0 < c <C$ such that 
$$
\lim_{n \to \infty} \Pr([n]_p \to (\ell_1, \dotsc, \ell_r)) = 
\begin{cases}
1, & p \geq C n^{-\frac{\ell_2}{\ell_1(\ell_2-1)}},\\
0, & p  \leq n^{-\frac{\ell_2}{\ell_1(\ell_2-1)}};
\end{cases}
$$
where for $A \subseteq[n]$ we write $A \to (\ell_1, \dotsc, \ell_r)$ to denote that $A$ has the property that for every $r$-colouring of $A$ there is a colour $i \in [r]$ admitting a monochromatic arithmetic progression of length $\ell_i$. While the $1$-statement of the result of Zohar~\cite{Zohar} is a special case of our main result, namely Theorem~\ref{thm:Rado}, the $0$-statement of the result of Zohar~\cite{Zohar} is of course not covered by our result. 
 An extension of the $0$-statement above to more general systems of Rado matrices could lead to the proof of Conjecture~\ref{conj:AHP}.

\bibliographystyle{amsplain_yk}
\bibliography{lit}
\end{document}